\documentclass[a4paper]{amsart}

\usepackage{amsmath, amssymb, amsthm, stmaryrd, verbatim}
\usepackage[utf8]{inputenc}
\usepackage[T1]{fontenc}
\usepackage[initials, alphabetic]{amsrefs}
\usepackage{mathtools}
\usepackage{enumitem}
\usepackage{xspace}
\usepackage{color, soul}
\usepackage[dvipsnames]{xcolor}
\usepackage{tikz}
\usepackage{caption}
\usepackage{tabularx}
\usepackage{nicefrac}

\usepackage{hyperref}

\makeatletter
\@ifpackageloaded{hyperref}{
\newcommand\myshade{85}
\colorlet{mylinkcolor}{violet}
\colorlet{mycitecolor}{YellowOrange}
\colorlet{myurlcolor}{Aquamarine}

\hypersetup{
  citecolor = mylinkcolor!\myshade!black,
  linkcolor  = mycitecolor!\myshade!black,
  urlcolor   = myurlcolor!\myshade!black,
  linkbordercolor = mylinkcolor!\myshade!black,
  colorlinks = true,
}

\def\namedlabel#1#2{\begingroup
    #2%
    \def\@currentlabel{#2}%
    \phantomsection\label{#1}\endgroup
}

}{
\newcommand{\texorpdfstring}[2]{#1}
\newcommand{\namedlabel}[2]{\label{#1}}
}
 \makeatother

\usepackage{cleveref}
\theoremstyle{plain}
\newtheorem{theorem}{Theorem}[section]
\newtheorem{lemma}[theorem]{Lemma}
\newtheorem{corollary}[theorem]{Corollary}
\newtheorem{proposition}[theorem]{Proposition}

\theoremstyle{definition}
\newtheorem{definition}[theorem]{Definition}
\theoremstyle{remark}
\newtheorem{question}[theorem]{Question}
\newtheorem*{remark*}{Remark}
\newtheorem{remark}[theorem]{Remark}

\newtheorem{claim}{Claim}[theorem]

\newcommand{\N}{\mathbb N}

\newcommand{\R}{\mathbb R}
\newcommand{\inter}{[0, 1]}

\newcommand{\Can}{2^{\N}}

\newcommand{\card}[1]{\lvert #1 \rvert}
\newcommand{\fhi}{\varphi}

\newcommand{\mL}{\mathcal{L}}
\newcommand{\F}{Fra\"iss\'e }
\newcommand{\bl}{\left[ \begin{smallmatrix}}
\newcommand{\br}{\end{smallmatrix} \right]}

\newcommand{\harpoone}[1]{%
	\begin{tikzpicture}[#1]%
	\draw (0, -0.5pt) -- (0, 5pt);%
	\draw (0, 5pt) -- (1.5pt, 3.5pt);%
	\end{tikzpicture}%
}%
\newcommand{\restr}[1]{\, _{\harpoone{}\, #1}}

\newcommand{\hide}[1]{}
\newcommand{\quot}[2]{{\raisebox{.1em}{$#1\!$}\left/\raisebox{-.1em}{$#2$}\right.}}

\newcommand{\symdif}{\mathbin{\triangle}}

\DeclarePairedDelimiter{\set}{\{}{\}}
\DeclarePairedDelimiter{\Card}{\lvert}{\rvert}
\DeclarePairedDelimiterX\setnew[2]{\{}{\}}{#1 \nonscript\;\delimsize \vert \nonscript\; #2}

\DeclareMathOperator{\diam}{diam}

\DeclareMathOperator{\homeo}{Homeo}
\DeclareMathOperator{\aut}{Aut}

\DeclareMathOperator{\interior}{int}
\DeclareMathOperator{\closure}{cl}
\newcommand{\binR}{\mathbin{R}}
\newcommand{\forests}{\mathcal F}%{\Pi_{\nabla}}
\newcommand{\diamonds}{\forests_{0}}
\newcommand{\prespace}{\mathbb F}
\newcommand{\fus}{F}
\newcommand{\fue}{\gamma}
\newcommand{\quotofpre}{\quot{\prespace}{R^{\prespace}}}
\newcommand{\altprespace}{\quot{\prespace}{\sim}}
\newcommand{\cappello}[1]{\llbracket #1 \rrbracket}

\newcommand{\mini}{m}
\newcommand{\maxi}{M}
\newcommand{\basis}{X}
\newcommand{\pwlin}[2]{P_{#1}^{#2}}
\newcommand{\dist}{d}

\DeclareMathOperator{\Exp}{\mathcal K}
\DeclareMathOperator{\branches}{MC}

\title{Fences, their endpoints, and projective Fraïssé theory}

\date{}
\author{Gianluca Basso}
\address{D\'epartement des Opérations \\ Universit\'e de Lausanne \\ Quartier UNIL-Chambronne B\^atiment Anthropole \\ 1015 Lausanne \\ Switzerland}
\curraddr{Institut Camille Jordan, Universit\'e Claude Bernard Lyon 1, Universit\'e de Lyon, 43, boulevard du 11 novembre 1918, 69622 Villeurbanne cedex, France}
\email{basso@math.univ-lyon1.fr}

\author{Riccardo Camerlo}
\address{Dipartimento di matematica \\ Universit\`a di Genova \\ Via Dodecaneso 35 \\ 16146 Ge\-no\-va\\ Italy}
\email{camerlo@dima.unige.it}


\begin{document}

\begin{abstract}
We introduce a new class of compact metrizable spaces, which we call fences, and its subclass of smooth fences.
We isolate two families $\forests, \diamonds$ of Hasse diagrams of finite partial orders and show that smooth fences are exactly the spaces which are approximated by projective sequences from $\diamonds$.
We investigate the combinatorial properties of Hasse diagrams of finite partial orders and show that $\forests, \diamonds$ are projective Fraïssé families with a common projective Fraïssé limit.
We study this limit and characterize the smooth fence obtained as its quotient, which we call a \F fence.
We show that the \F fence is a highly homogeneous space which shares several features with the Lelek fan, and we examine  the structure of its spaces of endpoints.
Along the way we establish some new facts in projective Fraïssé theory.
\end{abstract}

\maketitle

\section{Introduction}
In this paper we introduce and begin the study of a new class of topological spaces, which we call \emph{fences}.
These are the compact metrizable spaces whose connected components are either points or arcs.
Among them, we define  the subclass of smooth fences and characterize them as those fences admitting an embedding in $\Can \times [0, 1]$.

A major tool for our study are projective Fra\"iss\'e families of topological structures, for a given language $ \mathcal L $, and their limits --- called projective Fra\"iss\'e limits.
These were introduced by Irwin and Solecki in \cite{Irwin2006}.
In that paper, the authors focus on a particular example, where $ \mathcal L =\{ R\}$ contains a unique binary relation symbol such that its interpretation on the limit is an equivalence relation, and the quotient is a pseudo-arc.
The characterization of all spaces that can be obtained, up to homeomorphism, as quotients $ \quot{ \mathbb L }{R^{ \mathbb L }} $, where $( \mathbb L , R^{ \mathbb L})$ is the projective Fra\"iss\'e limit of a projective Fra\"iss\'e family of finite topological
$\set{R}$-structures
is settled in \cite{Camerl2010}.
In \cite{Basso} it is noted that, if we admit infinite languages, then every compact metrizable space can be obtained as such a quotient of a projective \F limit; some other examples for finite languages are also given.
In this article we provide a new example: we focus on a family $\forests$ of structures --- finite partial orders whose Hasse diagram is a forest --- which we show (\Cref{fraisseforests}) is projective Fraïssé; its limit $\prespace$ admits a  quotient $\quotofpre$ which is a smooth fence.
This space does not seem to appear in the literature and we call it the \emph{\F fence}.

We isolate a cofinal subclass $\diamonds$ of $\forests$ and we show that
smooth fences are exactly those spaces which are  quotients of projective limits of sequences from $\diamonds$ (\Cref{thmpidiamondfence,smoothfencespidiamond}).
This result creates a bridge between the combinatorial world and the topological one, which we exploit in \Cref{weakcharacthm} to obtain a characterization of the \F fence by isolating a topological property which yields the amalgamation property for $\diamonds$.

Our spaces, some of their properties, and the techniques we use have their analogs in the theory of fans.
A \emph{fan} is an arcwise connected and hereditarily unicoherent compact space that has at most one ramification point.
A fan with ramification point $t$ is \emph{smooth} if for any sequence $(x_{n})_{n \in \N}$ converging to $x$, the sequence $([t, x_{n}])_{n \in \N}$ of arcs connecting $t$ to $x_{n}$ converges to $[t, x]$.
Smooth fans where introduced in \cite{MR0227944} and have been extensively studied in continuum theory.
A point $x$ in a topological space $X$ is an \emph{endpoint} if whenever $x$ belongs to an arc $[a, b] \subseteq X$, then $x = a$ or $x = b$ (note that under this definition points whose connected component is a singleton are endpoints).
A \emph{Lelek fan} is a smooth fan with a dense set of endpoints.
Such a fan was first constructed in \cite{MR133806} and was later proven to be unique up to homeomorphism in \cite{MR1002079} and \cite{MR991691}.
In a series of papers (\cites{Bartos2015, Bartos2017, MR3939568}) Barto\v{s}ov\'{a}  and Kwiatkowska have studied the Lelek fan and the dynamics of its homeomorphism group by realizing it as a quotient of a projective Fraïssé limit of a particular class of ordered structures.

Besides the fact that both can be obtained as quotients of projective \F limits of some class of ordered structures, the \F fence and the Lelek fan share several other features:
\begin{itemize}
\item Both are as homogeneous as possible, namely they are $\nicefrac{1}{3}$-homogeneous (see \cite{MR3712972} for the Lelek fan and \Cref{onethirdhomo} for the \F fence).
\item
Both are universal in the respective classes with respect to embeddings that preserve endpoints (see \cite{Dijkstra2010} for the Lelek fan and \Cref{thm:FraisseFenceUniversal} for the \F fence).
\item For both, the set of endpoints is dense (see \Cref{allendpointsaredense} for the \F fence).
In fact, the Lelek fan is defined as the unique smooth fan with a dense set of endpoints; the \F fence too has a characterization in terms of denseness of endpoints (see \Cref{weakcharacthm}).
\item The set of endpoints of the Lelek fan is homeomorphic to the complete Erdős space (\cite{MR1391294}), a homogeneous, almost zero-dimensional, $1$-dimensional cohesive space.
Among the subspaces of the set of endpoints of the \F fence there is a homogenous, almost zero-dimensional, $1$-dimensional space $ \mathfrak M $ which is not cohesive (\Cref{riassunto}(iv)).
\end{itemize}
A space with the properties mentioned for $ \mathfrak M $ was constructed in \cite{Dijkstra2006} as a counterexample to a question by Dijkstra and van Mill.
This raises the question of whether the two examples are homeomorphic and whether they can be regarded as a non-cohesive analog of the complete Erdős space.

To obtain our results, we establish combinatorial criteria which are of general interest in the context of projective \F theory.
\Cref{lemmafour} characterizes which projective sequences of structures in a language containing a binary relation symbol $\set{R}$ have limit on which $R$ is an equivalence relation, and \Cref{singletonsdenseiffirreducible} gives conditions under which the resulting quotient map is irreducible.
The irreducibility condition entails a correspondence between structures in the projective sequence and regular quasi-partitions of the quotient, which in turn aids the combinatorial-topological translation.

Here is the plan of the paper.
We begin in \Cref{defandnot} with recalling some notions and proving some technical lemmas which will lay the basis of this work.
In \Cref{secmain} we introduce the topological structures that constitute the main combinatorial objects of our study, prove that the relevant classes $ \forests $ and $ \diamonds $ are projective Fraïssé and investigate the properties of the projective limits of $ \diamonds $.
We define fences and characterize smooth fences in \Cref{sec:fences}, where we also display the relation linking them to $ \diamonds $.
Finally in \Cref{sec:theprojlimit} we characterize topologically the quotient of the projective \F limit of $\forests$, explore its homogeneity and universality features and investigate its spaces of endpoints.

\subsection*{Acknowledgments}
We would like to thank Aleksandra Kwiatkowska and Lionel Nguyen Van Thé for their comments and suggestions.
We are also grateful to the anonymous referee for a careful reading of the paper and some useful suggestions.
The first author's work was conducted as a doctoral student at Univerisité de Lausanne and Università di Torino, and partially within the framework of the LABEX MILYON (ANR-10- LABX0070) of Université de Lyon, within the program ``Investissements d'Avenir'' (ANR-11- IDEX-0007) operated by the French National Research Agency (ANR).
The research of the second author was partially supported by PRIN 2017NWTM8R - ``Mathematical logic: models, sets, computability''.

\section{Basic terminology and definitions} \label{defandnot}
Let $X$ be a topological space.
If $A$ is a subset of $X$, then $\interior_{X}(A), \closure_{X}(A), \partial_{X}(A)$ denote the interior, closure, and boundary of $A$ in $X$, respectively.
We drop the subscript whenever the ambient space is clear from context.
A closed set is regular if it coincides with the closure of its interior.
We denote by $\Exp(X) = \setnew*{K \subseteq X}{K \text{ compact}}$ the space of compact subsets of $X$, with the Vietoris topology.
This is the topology generated by the sets $ \setnew{K\in \Exp (X)}{K\subseteq O} $ and $\setnew{K \in \Exp(X) }{K \cap O\neq \emptyset} $, for $O$ varying among the open subsets of $X$.
If $X$ is compact metrizable, so is $\Exp(X)$.
Let $\homeo(X)$ denote the group of homeomorphisms of $X$.

By \emph{mesh} of a covering of a metric space, we indicate the supremum of the diameters of its elements.

We collect here the definitions of some basic topological concepts we need.
\begin{definition}
\leavevmode
\begin{itemize}
\item A space is \emph{almost zero-dimensional} if each point has a neighborhood basis consisting of closed sets that are intersection of clopen sets.
\item A space is $X$ \emph{cohesive} if each point has a neighborhood which does not contain any nonempty clopen subset of $X$.
\item The \emph{quasi-component} of a point is the intersection of all its clopen neighborhoods.
A space is \emph{totally separated} if the quasi-component of each point is a singleton.
\item A space is $n$-\emph{homogeneous} if for every two sets of $n$ points there is a homeomorphism sending one onto the other.
\item A space $X$ is $\nicefrac{1}{n}$-\emph{homogeneous} if the action of $\homeo(X)$ on $X$ has exactly $n$ orbits.
\item A space is $h$-\emph{homogeneous} if it is homeomorphic to each of its nonempty clopen subsets.
\end{itemize}
\end{definition}

When we talk about dimension, we mean the inductive dimension.

\subsection{Topological structures}

We recall here some basic definitions, mainly from \cites{Irwin2006, Camerl2010}, sticking to relational first order languages, since we will not use other kinds of languages in this paper.

Let thus a relational first order language $ \mathcal L $ be given.
A \emph{topological} $ \mathcal L $-\emph{structure} is a zero-dimensional compact metrizable space that is also an $ \mathcal L $-structure such that the interpretations of the relation symbols are closed sets.
In particular, the topology on finite topological $\mathcal L $-structures is discrete.
We will usually suppress the word ``topological'' when referring to finite topological $\mathcal L $-structures.

An \emph{epimorphism} between topological $ \mathcal L $-structures $A, B$ is a continuous surjection $ \fhi :A\to B$ such that
\[
r^B=\underset{n \text{ times}}{\underbrace{ \fhi \times \ldots \times \fhi}}\, [r^A]
\]
for every $n$-ary relation symbol $r \in \mL$: in other words, $r^B(b_1,\ldots ,b_n)$ if and only if there exist $a_1,\ldots ,a_n\in A$ such that
\[
\fhi (a_1)=b_1,\ldots , \fhi (a_n)=b_n,\quad r^A(a_1,\ldots ,a_n).
\]
An \emph{isomorphism} is a bijective epimorphism, so in particular it is a homeomorphism between the supports.
An isomorphism of $A$ onto $A$ is an \emph{automorphism} and we denote by $\aut(A)$ the group of automorphisms of $A$.
An epimorphism $ \fhi :A\to B$ \emph{refines} a covering $ \mathcal U $ of $A$ if the preimage of any element of $B$ is included in some element of $ \mathcal U $.
If $ \mathcal G , \mathcal G'
 $ are families of topological structures such that $ \mathcal G'
 \subseteq \mathcal G $ and for all $A\in \mathcal G$ there exist $ B\in \mathcal G'
$ and an epimorphism $\fhi: B \to A $, we say that $ \mathcal G'
 $ is \emph{cofinal} in $ \mathcal G $.

A family $ \mathcal G $ of topological $ \mathcal L $-structures is a \emph{projective Fra\"iss\'e family} if the following properties hold:
\begin{enumerate}
\item[\namedlabel{itm:JPP}{(JPP)}]
(joint projection property) for every $A, B\in \mathcal G $ there are $C\in \mathcal G $ and epimorphisms $C\to A$, $C\to B$;
\item[\namedlabel{itm:AP}{(AP)}]  (amalgamation property) for every $A, B, C\in \mathcal G $ and epimorphisms $ \fhi_1:B\to A$, $ \fhi_2:C\to A$ there are $D\in \mathcal G $ and epimorphisms $\psi_1:D\to B$, $\psi_2:D\to C$ such that $ \fhi_1\psi_1= \fhi_2\psi_2$.
\end{enumerate}
Given a family $ \mathcal G $ of topological $ \mathcal L $-structures, a topological $ \mathcal L $-structure $ \mathbb L $ is a \emph{projective Fra\"iss\'e limit} of $ \mathcal G $ if the following hold:
\begin{enumerate}
\item[\namedlabel{itm:Lone}{(L1)}]
	\label{Lone}
(projective universality) for every $A\in \mathcal G $ there is some epimorphism $ \mathbb L \to A$;
\item[\namedlabel{itm:Ltwo}{(L2)}]
for any clopen covering $ \mathcal U $ of $ \mathbb L $ there are $A\in \mathcal G $ and an epimorphism $ \mathbb L \to A$ refining $ \mathcal U $.
\item[\namedlabel{itm:Lthree}{(L3)}]
	\label{Lthree}
(projective ultrahomogeneity) for every $A\in \mathcal G $ and epimorphisms $ \fhi_1, \fhi_2: \mathbb L \to A$ there exists an automorphism  $\psi \in \aut(\mathbb L )$ such that $ \fhi_2= \fhi_1\psi $.
\end{enumerate}

Note that in the original definition of a projective Fra\"iss\'e limit in \cite{Irwin2006} item \ref{itm:Ltwo} was replaced by a different but equivalent property.

If $ \mathcal G $ is a projective Fra\"iss\'e family of finite $ \mathcal L $-structures and $\mathbb L$ satisfies \ref{itm:Lone} and \ref{itm:Ltwo}, then \ref{itm:Lthree} holds if and only if the following extension property holds:
\begin{enumerate}
\item[\namedlabel{itm:Lthreeprime}{(L3$'$)}]
for any $A, B\in \mathcal G $ and epimorphisms $\varphi: B\to A$, $\psi : \mathbb L \to A$ there exists an epimorphism $\chi : \mathbb L \to B$ such that $\varphi\chi =\psi $.
\end{enumerate}
The proof is the same as in \cite{panagi2017}*{Lemma 3}.

In \cite{Irwin2006} it is proved that every nonempty, at most countable, projective Fra\"iss\'e family of finite $ \mathcal L $-structures has a projective Fra\"iss\'e limit, which is unique up to isomorphism.

If $ \mathcal G $ is a class of topological $ \mathcal L $-structures, a \emph{projective sequence} in $ \mathcal G $ is a sequence $(A_n, \fhi_n^m)_{n\in \N , m\ge n}$, where:
\begin{itemize}
\item $A_n\in \mathcal G $;
\item $ \fhi_n^{n+1}:A_{n+1}\to A_n$ is an epimorphism, for each $n\in \N $;
\item $ \fhi_n^m= \fhi_n^{n+1}\cdots \fhi_{m-1}^m:A_m\to A_n$ for $n<m$, and $ \fhi_n^n:A_n\to A_n$ is the identity.
\end{itemize}
The \emph{projective limit} for such a sequence is the topological $ \mathcal L $-structure $ \mathbb A $, whose universe is $ \mathbb A = \setnew{u\in \prod_{n\in \N }A_n}{\forall n\in \N \ u(n)= \fhi_n^{n+1}(u(n+1))} $ and such that $r^{ \mathbb A }(u_1, \ldots , u_j)\Leftrightarrow\forall n\in \N \ r^{A_n}(u_1(n), \ldots , u_j(n))$, for every $j$-ary relation symbol $r\in \mathcal L $.
We denote by $ \fhi_n: \mathbb A \to A_n$ the $n$-th projection map: this is an epimorphism.

A \emph{fundamental sequence} for $\mathcal G$ is a projective sequence $(A_n, \fhi_n^m)$  such that the following properties hold:
\begin{itemize}
\item[\namedlabel{itm:Fone}{(F1)}] $\{ A_n\}_{n\in \N }$ is cofinal in $\mathcal G$;
\item[\namedlabel{itm:Ftwo}{(F2)}]
for any $n$, any $A, B\in \mathcal G $ and any epimorphisms $ \theta_1 :B\to A$, $ \theta_2:A_n\to A$, there exist $m\geq n$ and an epimorphism $\psi :A_m\to B$ such that $ \theta_1\psi= \theta_2 \fhi_n^m$.
\end{itemize}
To study projective Fra\"iss\'e limits it is enough to consider fundamental sequences, due to the following fact whose details can be found in \cite{Camerl2010}.

\begin{proposition} \label{pFfpFlfs}
Let $ \mathcal G $ be a nonempty, at most countable (up to isomorphism) family of finite $ \mathcal L $-structures.
Then the following are equivalent.
\begin{enumerate}
\item $ \mathcal G $ is a projective Fra\"iss\'e family;
\item $ \mathcal G $ has a projective Fra\"iss\'e limit;
\item $ \mathcal G $ has a fundamental sequence.
\end{enumerate}
If these conditions hold and $ \mathcal G_0$ is cofinal in $ \mathcal G $ then $\mathcal G_{0}$ is a projective Fra\"iss\'e family and the projective Fra\"iss\'e limits of $\mathcal G _{0}, \mathcal G $, and of the fundamental sequence coincide.
A projective Fra\"iss\'e limit for them is the projective limit of the fundamental sequence.
\end{proposition}

If $\mathcal G$ is a projective Fraïssé family, one can check whether a given projective sequence is fundamental for $\mathcal G$ with the following.

\begin{proposition}\label{provefundseq}
Let $\mathcal G$ be a projective Fraïssé family of topological $\mL$-structures.
Let $(A_{n}, \fhi_{n}^m)$ be a projective sequence in $\mathcal G$.
Assume that for each $A\in \mathcal G $, $n \in \N $, and epimorphism $\theta:A \to A_{n}$, there exist $m\ge n$ and an epimorphism $\psi: A_{m} \to A$ such that $\theta\psi= \fhi^{m}_{n}$.
Then $(A_n, \fhi_n^m)$ is a fundamental sequence for $ \mathcal G $.
\end{proposition}
\begin{proof}
\ref{itm:Fone} Let $A\in \mathcal G $, by \ref{itm:JPP} there exist $A' \in \mathcal G$, and epimorphisms $\fhi: A' \to A$ and $\fhi': A' \to A_{0}$.
By hypothesis there are $n$ and an epimorphism $\theta:A_{n} \to A'$ such that $\fhi' \theta = \fhi^{n}_{0}$.
Then $\fhi \theta$ is an epimorphism $A_{n} \to A$, as wished.

\ref{itm:Ftwo} Let $A, B\in \mathcal G $ and epimorphisms $ \theta_1 :B\to A$, $ \theta_2:A_n\to A$.
By \ref{itm:AP} there exist $C\in \mathcal G$ and epimorphisms $\rho_{1}:C\to B$ and $\rho_{2}:C\to A_{n}$ such that $\theta_1\rho_{1}=\theta_{2}\rho_2$.
By hypothesis, there exist $m\ge n$ and an epimorphism $\psi': A_{m} \to C$ such that $\rho_{2}\psi'= \fhi^{m}_{n}$.
Then $\psi = \rho_{1}\psi':A_{m} \to B$ is such that $ \theta_{1}\psi= \theta_{2} \fhi^{m}_{n}$.
\end{proof}

Notice that the converse of \Cref{provefundseq} holds as well.

\subsection{Fine projective sequences} \label{fineprojectivesequences}

In the sequel, whenever we denote a language with a subscript, like in $ \mathcal L_R$, we mean that the language contains a distinguished binary relation symbol represented in the subscript.
The interpretation of $R$  in an $\mL_R$-topological structure is expected to be reflexive and symmetric.
These properties are preserved under projective limits.
A \emph{prespace} is any $ \mL_R$-topological structure $A$ where the interpretation of $R$ is also transitive, that is, an equivalence relation; in this case, we say that $A$ is a prespace of $ \quot{A}{R^A} $.
Since $R^{A}$ is a closed equivalence relation, the quotient map $p:A\to \quot{A}{R^A} $ is closed.
Notice that $ \quot{A}{R^A} $ is then endowed with an $( \mL_R\setminus \set{R} )$-structure, where $r^{A/R^A}=p\times\ldots\times p[r^A]$, for any $r\in \mL_R\setminus \set{R} $; all such relations are closed.

\begin{definition}
A projective sequence $(A_n, \fhi_n^m)$ of finite $\mL_{R}$-structures and epimorphisms is \emph{fine} whenever its projective limit is a prespace.
If $(A_{n}, \fhi_{n}^m)$ is a fine projective sequence in $ \mL_R$ with projective limit $ \mathbb A $ and $X$ is a compact metrizable space homeomorphic to $ \quot{ \mathbb A }{R^{ \mathbb A }} $, we say that  $(A_{n}, \fhi_{n}^m)$ \emph{approximates} $X$.
\end{definition}

Given a reflexive graph (that is, a reflexive and symmetric relation) $R$ on some set, denote by $d_R$ the distance on the graph, where $d_R(a, b)=\infty $ if $a, b$ belong to distinct connected components of the graph.
Note that if $R, S$ are reflexive graphs and $\varphi $ is a function between them such that $x\binR y\Rightarrow \varphi(x) \mathbin{S} \varphi(y)$ for all $x, y$, then the inequality $d_S(\varphi (x), \varphi (y))\le d_R(x, y)$ holds for every $x, y$.

We can determine whether a sequence is fine by checking that the $R$-distance of points which are not $R$-related tends to infinity. More precisely:

\begin{lemma} \label{lemmafour}
Let $(A_{n}, \fhi_{n}^m)$ be a projective sequence of finite $ \mathcal L_R$-structures, with projective limit $ \mathbb A $.
Assume that $R^{A_n}$ is reflexive and symmetric for every $n\in \N $.
The projective sequence is fine if and only if for all $n \in \N$ and $a, b \in A_{n}$ with $d_{R^{A_n}}(a, b) = 2$, there is $m>n$ such that if $a' \in (\fhi^{m}_{n})^{-1}(a), b' \in (\fhi^{m}_{n})^{-1}(b)$ then $d_{R^{A_m}}(a', b') \ge 3$.
\end{lemma}
\begin{proof}
Let $a, b \in A_{n}$ with $d_{R^{A_n}}(a, b) = 2$, say $a \binR^{A_n}c \binR^{A_n}b$.
If for each $m>n$ there are $a_{m} \in (\fhi^{m}_{n})^{-1}(a), b_{m} \in (\fhi^{m}_{n})^{-1}(b)$ with $d_{R^{A_m}}(a_{m}, b_{m}) = 2$, say $a_{m}\binR^{A_m}c_{m}\binR^{A_m}b_{m}$, let
\[
x_m\in \fhi_m^{-1}(a_m), \quad y_m\in \fhi_m^{-1}(b_m), \quad z_m, z'_m\in \fhi_m^{-1}(c_m),
\]
with $x_m\binR^{ \mathbb A }z_m, z'_m\binR^{ \mathbb A }y_m$.
Passing to a suitable subsequence, let
\[
x=\lim_{h\rightarrow\infty }x_{m_h}, \quad y=\lim_{h\rightarrow\infty }y_{m_h}, \quad z=\lim_{h\rightarrow\infty }z_{m_h}=\lim_{h\rightarrow\infty }z'_{m_h},
\]
so that $x\binR^{ \mathbb A }z\binR^{ \mathbb A }y$.
However, $x, y$ are not $R^{ \mathbb A }$-related (otherwise $a \binR^{A_n}b$), so $(A_{n}, \fhi_{n}^m)$ is not fine.

On the other hand, if $(A_{n}, \fhi_{n}^m)$ is not fine there are $x, y\in \mathbb A $ such that $d_{R^{ \mathbb A }}(x, y) = 2$, say $x\binR^{ \mathbb A }z\binR^{ \mathbb A }y$, for $x, y, z$ distinct points.
There is $n \in \N$ such that for all $m \ge n$ the points $\fhi_{m}(x), \fhi_m(y), \fhi_m(z)$ are  distinct and $\neg(\fhi_{m}(x) \binR^{A_m}\fhi_{m}(y))$, so $d_{R^{A_m}}(\fhi_{m}(x), \fhi_{m}(y)) = 2$.
Therefore the property does not hold for $\fhi_n(x), \fhi_n(y)$.
\end{proof}

\begin{definition} \label{rconnectedness}
Let $A$ be a topological $\mathcal L_{R}$-structure and $B \subseteq A$.
We say $B$ is $R$-\emph{connected} if for any two clopen sets $U, U' \subseteq A$ such that $U\cap B, U'\cap B$ partition $B$, there are $x \in U\cap B, x' \in U' \cap B$ such that $x \binR^Ax'$.
\end{definition}

Notice that if $A$ is a finite $\mathcal L_{R}$-structure and $R^A$ is symmetric, $R$-connectedness coincides with the usual notion of connectedness for the graph $R^{A}$.

\begin{lemma} \label{oldparagraph}
Let $A$ be a prespace.
Then the image of an $R$-connected closed subset $B \subseteq A $ under the quotient map $p: A \to  \quot{A }{R^{A }}$ is closed and connected.
\end{lemma}

\begin{proof}
The set $p[B]$ is closed as $p$ is a closed map.
If $p[B]$ were disconnected, let $C, C'$ be disjoint, nonempty, closed subsets of $ \quot{A }{R^{A }} $ such that $p[B]=C\cup C'$.
Then $p^{-1}(C)\cap B, p^{-1}(C')\cap B$ are disjoint, nonempty, closed subsets of $A $ whose union is $B$.
Let $U, U'$ be disjoint clopen subsets of $A $ with $p^{-1}(C)\cap B\subseteq U, p^{-1}(C')\cap B\subseteq U'$.
By the assumption, there are $u\in p^{-1}(C)\cap B, u'\in p^{-1}(C')\cap B$ with $u\binR^{A }u'$, contradicting the disjointness of $C, C'$.
\end{proof}

For the remainder of the section we fix a fine projective sequence of finite $ \mathcal L_R$-structures $(A_{n}, \fhi_{n}^m)$ with projective limit $ \mathbb A $ and with quotient map $p: \mathbb A \to \quot{ \mathbb A }{R^{ \mathbb A }} $.

\begin{lemma}\label{generalprop}
\leavevmode
\begin{enumerate}
\item The mesh of the sequence $\left(\setnew{ \fhi_n^{-1}(a)}{a \in A_{n}} \right)_{n \in \N}$ tends to $0$.
In particular, the sets $ \fhi_n^{-1}(a)$ for $n\in \N , a\in A_n$ form a basis for the topology of $ \mathbb A $.
\item The mesh of the sequence $\left(\setnew{p[ \fhi_n^{-1}(a)]}{a \in A_{n}} \right)_{n \in \N}$ tends to $0$.
\end{enumerate}
\end{lemma}
\begin{proof}
(1) Suppose that there is $\varepsilon > 0$ such that for infinitely many $n \in \N$, there is $a_{n} \in A_{n}$ with $\diam( \fhi_n^{-1}(a_n))\ge \varepsilon$.
Fix such $a_{n}$'s and consider the forest $T= \setnew{ \fhi_{n'}^n(a_n)}{n'<n} $, so that $ \diam ( \fhi_n^{-1}(b))\ge\varepsilon $ for every $b\in A_n$ in the forest.
Let $u=(b_0, b_1, \ldots )\in \mathbb A $ be an infinite branch in $T$.
Since
\[
n<n'\Rightarrow \fhi_{{n'}}^{-1}(b_{n'}) \subseteq \fhi_{n}^{-1}(b_{n})
\]
it follows that the sequence $ \fhi_n^{-1}(b_n)$ converges in $ \Exp ( \mathbb A )$ to $K=\bigcap_{n\in \N } \fhi_n^{-1}(b_n)$ with $\diam(K) \ge \varepsilon$.
But $\bigcap_{n\in \N } \fhi_n^{-1}(b_n)= \set{u} $, a contradiction.

(2) By (1) and the fact that function $p$ is uniformly continuous.
\end{proof}

\begin{lemma} \label{limitconnected}
If $B_{n} \subseteq A_{n}$, for $n \in \N$, are $R$-connected subsets and $(\fhi_{n}^{-1}(B_{n}))_{n \in \N}$ converges in $\Exp( \mathbb A )$ to $K$, then $K$ is $R$-connected.
\end{lemma}
\begin{proof}
Let $U, U'$ be clopen, nonempty subsets of $ \mathbb A $, with some positive distance $\delta$, such that $U\cap K, U'\cap K $ partition $K$.
Consider the open neighborhood $O = \setnew{ C \in \Exp( \mathbb A )}{C \subseteq U \cup U', C \cap U \neq \emptyset, C \cap U' \neq \emptyset}$ of $K$ in $\Exp(\mathbb A)$.
Let $n \in \N$ be such that $ \fhi_{n}^{-1}(B_{n}) \in O$, and $\diam( \fhi_n^{-1}(a))<\delta$ for each $a \in A_{n}$: such a $n$ exists by \Cref{generalprop}.
Then each $ \fhi_{n}^{-1}(a)$ for $a \in B_{n}$ is either contained in $U$ or in $U'$, as the distance between the two clopen sets is greater than $ \diam ( \fhi_n^{-1}(a))$, and $U, U'$ each contain at least one such set, since $ \fhi_n^{-1}(B_{n})$ has nonempty intersection with both $U$ and $U'$.
It follows that $\fhi_n[U]\cap B_n, \fhi_n[U']\cap B_n$ partition $B_{n}$.
But $B_{n}$ is $R$-connected, so there are $a \in B_{n} \cap \fhi_n[U], a' \in B_{n} \cap \fhi_n[U']$ such that $a \binR^{A_n}a'$, and thus there exist $x \in \fhi_n^{-1}(a) \subseteq U, x' \in \fhi_n^{-1}(a') \subseteq U'$ such that $x \binR^{ \mathbb A }x'$.
So $K$ is $R$-connected.
\end{proof}

\begin{corollary} \label{quotlimitconnected}
If $B_{n} \subseteq A_{n}$ are $R$-connected subsets and $(p[\fhi^{-1}_{n}(B_{n})])_{n \in \N}$ converges in $\Exp( \quot{ \mathbb A }{ R^{ \mathbb A }} )$ to some $K$, then $K$ is connected.
\end{corollary}

\begin{proof}
Let $n_k$ be an increasing sequence of natural numbers such that $ \fhi_{n_k}^{-1}(B_{n_k})$ converges in $ \Exp ( \mathbb A )$, say $\lim_{k\rightarrow\infty } \fhi_{n_k}^{-1}(B_{n_k})=L$.
Then \[
\lim_{n\rightarrow\infty }p[ \fhi_n^{-1}(B_n)]= \lim_{k\rightarrow\infty }p[ \fhi_{n_k}^{-1}(B_{n_k})]=p[L],
\]
whence $K=p[L]$.
Now apply \Cref{limitconnected,oldparagraph}.
\end{proof}

\subsection{Irreducible functions and regular quasi-partitions}

Given topological spaces $X$, $Y$, a continuous map $f:X\to Y$ is \emph{irreducible} if $f[K] \neq Y$ for all proper closed subsets $K \subset X$.

We recall some basic results on irreducible closed surjective maps between compact metrizable spaces, whose proofs can be found in \cite{MR785749}.
Let $f:X\to Y$ be such a map.
Given $A\subseteq X$, let $f^{\#}(A) = \setnew{y\in Y}{f^{-1}(y) \subseteq A}$.
If $O\subseteq X$ is an open set, then $f^{\#}(O)$ is open and $f^{-1}(f^{\#}(O))$ is dense in $O$.
If $C \subseteq X$ is a regular closed set, then $C = \closure(f^{-1}(f^{\#}(\interior(C))))$, and $f[C]= \closure(f^{\#}(\interior(C)))$, so in particular the image of a regular closed set is regular.
The preimage of any point by $f$ is either an isolated point or has empty interior.
If $C, C'$ are regular closed and $f[C] = f[C']$ then $C = C'$;
if $\interior(C \cap C') = \emptyset$ then $\interior(f[C] \cap f[C']) = \emptyset$.

\begin{definition}
	\label{def:regular quasi-partition}
A covering $ \mathcal C $ of a topological space is a \emph{regular quasi-partition} if the elements of $ \mathcal C $ are nonempty, regular closed sets and $\forall A, B\in \mathcal C \ (A\ne B\Rightarrow A\cap B\subseteq\partial (A)\cap\partial (B))$.
\end{definition}

\begin{lemma}
	\label{quasi-part}
If $X, Y$ are compact metrizable spaces and $f:X\to Y$ is an irreducible closed surjective map, then the image $f \mathcal C = \setnew{f[C] }{C \in \mathcal C}$ of a regular quasi-partition $\mathcal C$ of $X$ is a regular quasi-partition of $Y$, and the map $C \mapsto f[C]$ is a bijection between $\mathcal C$ and $f \mathcal C$.
\end{lemma}

\begin{proof}
The fact that $C\mapsto f[C]$ is a bijection is one of the basic properties of irreducible closed surjective maps between compact metrizable spaces.
The same for the fact that each $f[C]$ is a regular closed set.

Assume now that $C, C'\in \mathcal C $, and let $y\in f[C]\cap f[C']$.
We show that $y\notin \interior (f[C])$, and similarly $y\notin \interior (f[C'])$.
If toward contradiction $y\in \interior (f[C])$, let $O$ be open with $y\in O\subseteq f[C]$.
Since $y\in f[C']$ and $f[C']$ is regular closed, there is $y'\in O\cap \interior (f[C'])$, so that there exists an open set $V$ with $y'\in V\subseteq f[C]\cap f[C']$.
It follows that $ \interior (f[C]\cap f[C'])\ne\emptyset$, whence $ \interior (C\cap C')\ne\emptyset $, by irreducibility of $f$, and then $ \interior (C)\cap \interior (C')\ne\emptyset $, against $ \mathcal C $ being a regular quasi-partition.
\end{proof}

Recall that we have fixed a fine projective sequence of finite $ \mathcal L_R$-structures $(A_{n}, \fhi_{n}^m)$ with projective limit $ \mathbb A $ and with quotient map $p: \mathbb A \to \quot{ \mathbb A }{R^{ \mathbb A }} $.

\begin{lemma} \label{singletonsdenseiffirreducible}
The following are equivalent:
\begin{enumerate}
\item The set $M$ of points of $ \mathbb A $ whose $R^{ \mathbb A }$-equivalence class is a singleton is dense.
\item For each $n \in \N$ and $a \in A_{n}$ there are $m>n$ and $b \in A_{m}$ such that if $b' \binR^{A_m}b$ then $\fhi^{m}_{n}(b')=a$.
\item The quotient map $p: \mathbb A \to \quot{ \mathbb A }{R^{ \mathbb A }} $ is irreducible.
\end{enumerate}
\end{lemma}

\begin{proof}
$(1)\Rightarrow (3)$.
Let $K \subset \mathbb A $ be a proper closed subset.
Then there is $x \in M\setminus K$, so that $p(x)\notin p[K]$.
Thus $p$ is irreducible.

$(3)\Rightarrow (2)$.
Let $n \in \N$ and $a \in A_{n}$.
By irreducibility of $p$,
\[
O=p^{-1}(p^{\#}(\fhi_{n}^{-1}(a)))= \setnew{x\in \mathbb A }{[x]_{R^{ \mathbb A }}\subseteq \fhi_n^{-1}(a)}
\]
is an open, nonempty, and $R^{ \mathbb A }$-invariant set contained in $ \fhi_{n}^{-1}(a)$.
Let $m>n$ and $b \in A_{m}$ be such that $ \fhi_{m}^{-1}(b) \subseteq O$, which exist since such sets are a basis for the topology on $ \mathbb A $.
If $b' \binR^{A_m}b$, there are $x \in \fhi_{m}^{-1}(b), x' \in \fhi_{m}^{-1}(b')$ such that $x \binR^{ \mathbb A }x'$.
But $x \in \fhi_{m}^{-1}(b) \subseteq O$, which is $R^{ \mathbb A }$-invariant, so also $x' \in O$.
It follows that $ \fhi_{n}(x') = a$ and thus $\fhi^{m}_{n}(b') = a$, for $\fhi_{n}= \fhi^{m}_{n}\fhi_{m}$.

$(2)\Rightarrow (1)$.
Since $\setnew{\fhi_{n}^{-1}(a)}{n \in \N, a \in A_{n}}$ is a basis for the topology on $ \mathbb A $ it suffices to fix $n \in \N$ and $a \in A_{n}$ and prove that there is $x \in M$ with $\fhi_n(x) = a$.
We construct a sequence $n_i$ and elements $b_{i}\in A_{n_i}$ by induction.
Let $n_{0} = n$ and $b_{0} = a$.
Given $b_{i} \in A_{n_{i}}$, by hypothesis there are $m>n_{i}$ and $b \in A_{m}$ such that whenever $b' \binR^{A_m}b$ it follows that $\fhi^{m}_{n_{i}}(b')=b_{i}$.
Set $n_{i+1} = m$ and $b_{i+1} = b$.
Thus $\fhi^{n_{i+1}}_{n_{i}}(b_{i+1}) = b_{i}$ for each $i$, so there exists $x \in \mathbb A $ such that $\fhi_{n_{i}} (x) = b_{i}$, for each $i \in \N$.
In particular $\fhi_{n}(x) = a$.
Let $y \binR^{ \mathbb A }x$; if towards contradiction $y \neq x$ then there is $i \in \N$ such that $\fhi_{n_{i}}(y) \neq \fhi_{n_{i}}(x) = b_{i}$.
But $\fhi_{n_{i+1}}(y) \binR^{A_{n_{i+1}}}\fhi_{n_{i+1}}(x)= b_{i+1}$, so $ \fhi_{n_{i}}(y)= \fhi^{n_{i+1}}_{n_{i}} \fhi_{n_{i+1}}(y)= b_{i}$ by construction of $b_{i+1}$, a contradiction.
\end{proof}

If $ \fhi : \mathbb A \to A$ is an epimorphism onto a finite $ \mL_R$-structure $A$ and $a\in A$, we let
\[
\cappello{a}_{ \fhi }=p[ \fhi^{-1}(a)], \qquad \cappello{A}_{ \fhi }= \setnew{ \cappello{a}_{ \fhi }}{a\in A} .
\]

If the quotient map $p: \mathbb A \to \quot{ \mathbb A }{R^{ \mathbb A }}$ is irreducible, then $ \cappello {A}_{ \fhi }$ is a regular quasi-partition of $ \quot{ \mathbb A }{R^{ \mathbb A }} $ by \Cref{quasi-part}, and the function
\[
a \in A\mapsto \cappello{a}_{ \fhi }\in \cappello{A}_{ \fhi }
\]
is a bijection.

\begin{lemma} \label{clmboundary}
Suppose that the quotient map $p: \mathbb A \to \quot{ \mathbb A }{R^{ \mathbb A }}$ is irreducible.
For every $n \in \N$, $a\in A_{n}$,
\begin{multline*}
\partial ( \cappello{a}_{ \fhi_n}) = \setnew{x \in \cappello{a}_{\fhi_n}}{ \exists a'\ne a, a' \binR^{A_n}a, x \in \cappello{a'}_{\fhi_n}} = \\
=\setnew{x\in \cappello{a}_{\fhi_n}}{\exists a'\ne a, x\in \cappello{a'}_{\fhi_n}} .
\end{multline*}
Moreover, regardless of the irreducibility of $p$, if $p$ is at most $2$-to-$1$ then for each $x$ there are at most two $a\in A_n$ such that $x\in \cappello{a}_{ \fhi_n}$.
\end{lemma}

\begin{proof}
Let $x\in\partial ( \cappello{a}_{ \fhi_n})$, so that $x = p(u)$ for some $u \in \fhi_{n}^{-1}(a)$.
As each $ \cappello{a'}_{ \fhi_n}$ is closed, this implies that there exists $a'\in A_{n}, a'\ne a$ such that $x \in \cappello{a'}_{ \fhi_n}$, so that there is $v\in \fhi_{n}^{-1}(a')$ with $u \binR^{ \mathbb A }v$; in turns, this entails that $a \binR^{A_n}a'$.

Let now $x\in \cappello{a}_{ \fhi_n}$, and assume that there exists $a'\in A_{n}$, with $a'\ne a, x\in \cappello{a'}_{ \fhi_n}$.
Since $ \cappello{a}_{ \fhi_n}\cap \cappello{a'}_{ \fhi_n}\subseteq\partial ( \cappello{a}_{ \fhi_n})\cap\partial ( \cappello{a'}_{ \fhi_n})$, it follows that $x\in\partial ( \cappello{a}_{ \fhi_n})$.

The last statement is a direct consequence of the definition of $ \cappello{a}_{ \fhi_n}$.
\end{proof}

\section{Finite Hasse forests} \label{secmain}

Henceforth fix $\mL_R= \set{R, \le}$, where $\le$ is a binary relation symbol.
A \emph{Hasse partial order} (\emph{HPO}) is a topological $ \mL_R$-structure $P$ such that
\begin{itemize}
\item $\le^P$ is a partial order, that is, it is reflexive, anti-symmetric and transitive;
\item $a \mathbin{R^P} b$ if and only if $a=b$ or $a, b$ are one the immediate $\le^P$-successor of the other, that is:
	\begin{itemize}
	\item $a \le^P b $ and whenever $a\le^P c \le^P b$ it holds that $c= a$ or $c=b$; or
	\item $b \le^P a $ and whenever $b \le^P c \le^P a$ it holds that $c= a$ or $c=b$.
	\end{itemize}
\end{itemize}
Indeed, if $P$ is a HPO, the relation $R^P$ is the Hasse diagram of $\le^P$.
Where clear we shall write $a\le b$ instead of $a\le^Pb$, and similarly for $a < b$ and $a\binR b$.
When $a\le b$ we also let $[a, b]=\{ c\in P\mid a\le c\le b\} $.
If $\le^P$ is total, we say that $P$ is a \emph{Hasse linear order} (or \emph{HLO}).

If $P, P'$ are HPOs we denote by $P\sqcup P'$ the HPO where the support and the interpretations of $\le $ and $R$ are the disjoint unions of the corresponding notions in $P, P'$.

\begin{definition}
A \emph{Hasse forest} (\emph{H-forest}) is a HPO whose Hasse diagram has no cycles, and we denote by $ \forests $ the family of all finite H-forests.
\end{definition}

\begin{definition}
For an HPO $P$, denote by $ \branches(P)$ the set of maximal chains of $P$ with respect to the partial order $\le^P$.
\end{definition}

Notice that if $P\in \forests $ and $B\in \branches(P)$ then $B$ is the unique maximal chain to which both $\min B$ and $\max B$ belong.
Indeed, if $B' \in \branches(P)$ is such that $\min B, \max B \in B'$ then $\min B' = \min B$ and $\max B' = \max B$ by the maximality of $B$, so if $B\neq B'$ there would be two $R^P$-paths joining $\min B$ and $\max B$.

In \cite{Bartos2015} it is shown\footnote{Albeit with a different language, it is easy to see that a continuous surjection is an epimorphism with one such language iff it is so with the other, thus ensuring that the limit is the same.} that the class of all finite H-forests with a minimum is a projective \F family whose limit's quotient with respect to $R$ is the Lelek fan.
In \cite{Basso} it is shown that the class of all finite HLOs is a projective \F family whose limit's quotient is the arc.
Here we prove that, though the family of all finite HPOs is not a projective \F family, the family of all finite H-forests is.

We begin by describing a smaller yet cofinal family which plays a central role in the rest of the paper.
\begin{definition}
	\label{def: pidiamond}
Let $\diamonds$ be the collection of all $P\in \forests $ whose maximal chains are pairwise disjoint.
In other words, the elements of $\diamonds$ are the finite disjoint unions of finite HLOs.
\end{definition}

Notice that if $P \in \diamonds$ and $Q \subseteq P$ is $\le^P$-convex --- that is, whenever $b, b' \in Q$ and $a \in P$ are such that $b \le^{P} a \le^{P} b'$, then $a \in Q$ --- then $Q$ with the induced $\mL_{R}$-structure is in $\diamonds$.

\begin{proposition} \label{propcof}
$\diamonds$ is cofinal in the family of all finite HPOs.
\end{proposition}

\begin{proof}
Let $P$ be a finite HPO.
If $ \branches(P)=\{ B_1, \ldots , B_m\} $, let $P'=B'_1\sqcup\ldots\sqcup B'_m$ where every $B'_j$ is isomorphic to $B_j$ with the induced structure.
Then there is an epimorphism $\varphi: P' \to P$, given by letting $\varphi $ be an isomorphism from $B'_j$ onto $B_j$ for $1\le j\le m$.
\end{proof}

\begin{proposition}
The family of all finite HPOs is not a projective Fraïssé family.
\end{proposition}

\begin{proof}
We show that the family of all finite HPOs lacks amalgamation.
Let
\begin{align*}
S= & \{ a, b, c, d\}, \\
P= & \{ a_0, b_0, b'_0, c_0, d_0\}, \\
Q= & \{ a_1, b_1, c_1, c'_1, d_1\},
\end{align*}
be ordered as follows (see \Cref{fig:SPQ}).
\begin{itemize}
\item For $S$: $a=\min S, d=\max S$, and $b, c$ are incomparable.
\item For $P$: $a_0<b_0, a_0<c_0<d_0, b'_0<d_0$, and no other order comparabilities hold, except for reflexivity and transitivity.
\item For $Q$: $a_1<b_1<d_1, a_1<c_1, c'_1<d_1$, and no other order comparabilities hold, except for reflexivity and transitivity.
\end{itemize}

\begin{figure}[ht]
\centering
    \begin{tabularx}{\textwidth}{*{3}{>{\centering\arraybackslash}X}}
\begin{tikzpicture}[
every edge/.style = {draw=black, thick},
 vrtx/.style args = {#1/#2}{%
      circle, draw, fill=black, inner sep=0pt,
      minimum size=2mm, label=#1:#2}
                    ]
\node (A) [vrtx=right/$a$]     at ( 0, 0) {};
\node (B) [vrtx=left/$b$]     at (-1, 1) {};
\node (C) [vrtx=right/$c$]    at ( 1, 1) {};
\node (D) [vrtx=right/$d$]    at ( 0,2) {};
\path   (A) edge (B)
        (A) edge (C)
        (B) edge (D)
        (C) edge (D);
\end{tikzpicture}
\caption*{$S$}
&
\begin{tikzpicture}[
every edge/.style = {draw=black, thick},
 vrtx/.style args = {#1/#2}{%
      circle, draw, fill=black, inner sep=0pt,
      minimum size=2mm, label=#1:#2}
                    ]
\node (A) [vrtx=right/$a_{0}$]     at ( 0, 0) {};
\node (B) [vrtx=left/$b_{0}$]     at (-0.72, 0.9) {};
\node (B1) [vrtx=left/$b'_{0}$]     at (-0.72, 1.6) {};
\node (C) [vrtx=right/$c_{0}$]    at ( 1, 1.25) {};
\node (D) [vrtx=right/$d_{0}$]    at ( 0,2.5) {};

\path   (A) edge (B)
        (A) edge (C)
        (B1) edge (D)
        (C) edge (D);
\end{tikzpicture}
\caption*{$P$}
&
\begin{tikzpicture}[
every edge/.style = {draw=black, thick},
 vrtx/.style args = {#1/#2}{%
      circle, draw, fill=black, inner sep=0pt,
      minimum size=2mm, label=#1:#2}
                    ]
\node (A) [vrtx=right/$a_1$]     at ( 0, 0) {};
\node (B) [vrtx=left/$b_1$]     at ( -1, 1.25) {};
\node (C) [vrtx=right/$c_1$]     at (0.72, 0.9) {};
\node (C1) [vrtx=right/$c'_1$]    at ( 0.72, 1.6) {};
\node (D) [vrtx=right/$d_1$]    at ( 0,2.5) {};

\path   (A) edge (B)
        (A) edge (C)
        (B) edge (D)
        (C1) edge (D);
\end{tikzpicture}
\caption*{$Q$}
    \end{tabularx}
\caption{}
\label{fig:SPQ}
\end{figure}

Define $\varphi :P\to S, \psi :Q\to S$ by letting:
\begin{align*}
 & \varphi (a_0)=\psi (a_1)=a, \\
 & \varphi (b_0)=\varphi (b'_0)=\psi (b_1)=b, \\
 & \varphi (c_0)=\psi (c_1)=\psi (c'_1)=c, \\
 & \varphi (d_0)=\psi (d_1)=d.
\end{align*}
Then $\varphi , \psi $ are epimorphisms.
To show that there is no amalgamation, by \Cref{propcof} it is enough to show that there is no $F \in\diamonds$ with epimorphisms $\theta: F\to P, \rho :F\to Q$ such that $\varphi\theta =\psi\rho $.
Otherwise, as $a_0<d_0$, there must be $B\in \branches(F) $ and $i, i'\in B$, with $i<i'$, such that $\theta (i)=a_0, \theta (i')=d_0$, so that $\theta [B]=\{ a_0, c_0, d_0\} $; moreover $\rho (i)=a_1, \rho (i')=d_1$.
If $ j \in B$ is such that $\theta ( j )=c_0$, then $i< j <i'$ and $\rho ( j )\in\{ c_1, c'_1\} $, since $\fhi \theta = \psi \rho$.
If $\rho ( j )=c_1$, this contradicts $ j \le i'$, as $\rho( j ) \not \leq \rho(i')$; similarly, if $\rho ( j )=c'_1$, this contradicts $i\le j $.
\end{proof}

Let us turn to the proof of the central result of the section.

\begin{theorem}\label{fraisseforests}
The family $\forests$ of all finite H-forests is a projective \F family.
\end{theorem}

First, we note the following simple but useful observation.

\begin{lemma} \label{lemchain}
Let $P, P'\in \forests $, and let $\varphi :P\to P'$ be an epimorphism.
If  $B\in \branches(P)$, then there is $B'\in \branches(P')$ such that $\varphi [B]\subseteq B'$.
If $B'\in \branches(P')$, then there exists $B\in \branches(P)$ such that $\varphi [B]=B'$.
\end{lemma}

\begin{proof}
For the first statement, since $B\in \branches(P)$ and $\varphi $ is an epimorphism, then $\varphi [B]$ is a chain in $P'$, so $\varphi [B]$ is included in a maximal chain.

For the second assertion, fix $B'\in \branches(P')$.
Since $\min B'\le\max B'$ and $\varphi $ is an epimorphism, there are $a, b\in P$ such that $a\le b, \varphi (a)=\min B', \varphi (b)=\max B'$.
Let $B \in \branches (P)$ contain $a, b$. Since $\min B \leq a$ then $\fhi (\min B) \leq \min B'$, so $\fhi (\min B) = \min B'$; analogously, $\fhi(\max B) = \max B'$.
Since $P'$ is an H-forest and $ \fhi $ respects $R$, it follows that $\varphi[B]=B'$.
\end{proof}

We can also prove a sort of converse.
Given $ \mL_R$-structures $P, P'$ and a function $ \fhi :P\to P'$, we say that $ \fhi $ is $ \mL_R$-\emph{preserving} if $a\binR^Pb\Rightarrow \fhi(a)\binR^{P'} \fhi(b)$ and $a\leq^Pb\Rightarrow \fhi(a) \leq^{P'}\fhi(b)$, for every $a, b\in P$.

\begin{lemma} \label{iffchain}
Let $P, P'\in \forests$, and let $\varphi :P\to P'$ be an $ \mL_R$-preserving function.
If for each $B'\in \branches(P')$ there exists $B\in \branches(P)$ such that $\varphi [B]=B'$, then $\fhi$ is an epimorphism.
\end{lemma}

\begin{proof}
The function $\fhi$ is clearly surjective.
Let $a', b' \in P'$ be such that $a'\leq b'$ and let $B' \in \branches (P')$ with $a', b'\in B'$.
Let $B\in \branches(P)$ such that $\varphi [B]=B'$, then there are $a, b\in B$ such that $\fhi(a)=a', \fhi(b) =b'$ and $a \leq b$.
If $a' \binR b'$ with $a'<b'$, then $a, b$ can be chosen to be $R^P$-related by letting $a=\max (B \cap \fhi^{-1}(a'))$ and $b= \min (B \cap \fhi^{-1}(b'))$.
\end{proof}

\begin{proof}[Proof of \Cref{fraisseforests}]
Since for every $P \in \forests$ there is an epimorphism from $P$ to the H-forest consisting of a single point, it suffices to prove amalgamation. Let $P, Q, S \in \forests$ and epimorphisms $\varphi: P \to S$, $\psi: Q \to S$ be given.

For each $C \in \branches(P)$, by \Cref{lemchain} there is $D \in \branches(Q) $ such that $\psi[D] \supseteq \fhi[C]$.
Let $C' = \psi^{-1}(\fhi[C])\cap D$.
Since $C, \fhi[C], C'$ with the inherited relations are finite HLOs and $ \fhi \restr{C} , \psi\restr{C'} $ are, in particular, epimorphisms onto $ \fhi [C]$, by \ref{itm:AP} for HLOs \cite{Basso}*{Lemma 10} there exist a finite HLO $E_{C}$ and epimorphisms $\fhi'_{C}: E_{C} \to C$, $\psi'_{C} : E_{C} \to C'$ such that $ \fhi \restr{C} \fhi'_{C} = \psi \restr{C'} \psi'_{C}$.

Analogously, for each $C \in \branches(Q)$ there exists $D \in \branches (P)$ such that $\fhi[D] \supseteq \psi[C]$.
As above there exist a finite HLO $E_{C}$ and epimorphisms $\fhi'_{C} : E_{C} \to C'=\fhi^{-1}(\psi[C]) \cap D$ and $\psi'_{C}: E_{C} \to C$ such that $ \fhi \restr{C'} \fhi'_{C} = \psi \restr{C} \psi'_{C}$.

Define the $\mL_{R}$-structure:
\[
T = \bigsqcup\setnew*{E_{C}}{C \in \branches(P) \sqcup \branches(Q)} \in\diamonds,
\]
and $\fhi': T \to P, \psi': T \to Q$, where, for $x \in E_{C}$, $\fhi'(x) = \fhi'_{C}(x)$ and $\psi'(x) = \psi'_{C}(x)$.
By construction $\fhi \fhi' = \psi\psi' $.
Since $\fhi'_{C}, \psi'_{C}$ are epimorphisms then $\fhi', \psi'$ are $ \mL_R$-preserving.
Let $C \in \branches(P)$, then $\fhi'[E_{C}] = \fhi'_{C}[E_{C}] = C$.
Analogously if $C \in \branches(Q)$, then $\psi'[E_{C}] = \psi'_{C}[E_{C}] = C$.
By \Cref{iffchain}, $\fhi', \psi'$ are thus epimorphisms.
\end{proof}

By \Cref{fraisseforests}, \Cref{propcof} and \Cref{pFfpFlfs} it follows that:

\begin{corollary} \label{pfforests}
$\diamonds$ is a projective \F family with the same projective \F limit as $ \forests $.
\end{corollary}

\subsection{Projective limits of sequences in \texorpdfstring{$\diamonds$}{F0} }
In the next section we determine the spaces which are approximable by fine projective sequences from $\diamonds$.
For this, we establish some properties of projective sequences in $\diamonds$ and their limits which are of use later.
For the remainder of the section let $(P_n, \fhi_n^m)$ be a fine projective sequence in $\diamonds$ with projective limit $ \mathbb P $, and $p: \mathbb P \to \quot{ \mathbb P }{R^{ \mathbb P }} $ be the quotient map.
Notice that $\le^{ \mathbb P }$ is an order relation.

\begin{lemma} \label{connectedintervals}
Let $u, v\in \mathbb P $ with  $u\le v$.
Then $[u, v]$ is $R$-connected.
\end{lemma}
\begin{proof}
First notice that the sequence $ \fhi_n^{-1}([ \fhi_n(u), \fhi_n(v)])$ converges in $ \Exp ( \mathbb P )$ to $[u, v]$, since $\forall n\in \N $ $ \fhi_{n+1}^{-1}([ \fhi_{n+1}(u), \fhi_{n+1}(v)])\subseteq \fhi_n^{-1}([ \fhi_n(u), \fhi_n(v)])$ and
\[
\bigcap_{n\in \N } \fhi_n^{-1}([ \fhi_n(u), \fhi_n(v)])=[u, v].
\]
By \Cref{limitconnected} it is now enough to observe that every $[ \fhi_n(u), \fhi_n(v)]$ is $R$-connected.
\end{proof}

\begin{lemma} \label{fineclassesconvex}
The $R^{ \mathbb P }$-equivalence classes contain at most two elements; moreover, each class is totally ordered and convex with respect to $\le^{ \mathbb P }$.
\end{lemma}

\begin{proof}
Let $u, v, w\in \mathbb P $ be $R^{ \mathbb P }$-related elements.
If $u, v, w$ were all distinct, there would exist $n\in \N $ such that $ \fhi_n(u), \fhi_n(v), \fhi_n(w)$ are all distinct and pairwise $R^{P_n}$-related, which is impossible, since $P_{n} \in \diamonds$.

If $u\binR v$, then $ \fhi_n(u)\binR \fhi_n(v)$ for every $n$; in particular, $ \fhi_n(u), \fhi_n(v)$ are $\le^{P_n}$ comparable for every $n$.
It follows that either $\forall n\in \N \ \fhi_n(u)\le \fhi_n(v)$ or $\forall n\in \N \ \fhi_n(v)\le \fhi_n(u)$, whence either $u\le v$ or $v\le u$.

Finally, if $u\binR v$ but $u<w<v$ for some $u, v, w\in \mathbb P $, let $n\in \N $ be such that $ \fhi_n(u), \fhi_n(v), \fhi_n(w)$ are distinct.
Then both $ \fhi_n(u)\binR  \fhi_n(v)$ and $ \fhi_n(u)< \fhi_n(w)< \fhi_n(v)$, which is a contradiction.
\end{proof}

\begin{lemma}\label{alleqclassless}
If $u, v \in \mathbb P$ are not $R^{ \mathbb P }$-related and $u \le v$ holds, then whenever $u' \binR  u, v' \binR  v$, the relation $u' \le v'$ holds.
\end{lemma}
\begin{proof}
For $n \in \N$ big enough, $\fhi_{n}(u), \fhi_{n}(v)$ are distinct and not $R^{P_n}$-related.
Since $\fhi_{n}(u) \le \fhi_{n}(v)$, $P_{n} \in \diamonds$, and $R^{P_n}$-related distinct elements are one the immediate $\le^{P_n}$-successor of the other and viceversa, it follows that $\fhi_{n}(u') \le \fhi_{n}(v')$.
This inequality holding eventually, the relation $u'\le v'$ is established.
\end{proof}

\begin{corollary} \label{corollarytwentyfour}
The relation $\le^{\mathbb P/R^{\mathbb P}} =p\times p[\le^{ \mathbb P }]$ on $ \quot{ \mathbb P }{R^{ \mathbb P }} $ defined by letting $x\le^{\mathbb P/R^{\mathbb P}} y$ if there are $u\in p^{-1}(x), v\in p^{-1}(y)$ with $u\le v$, is a closed order relation.
\end{corollary}

\begin{proof}
That $\le^{ \mathbb P /R^{ \mathbb P }}$ is closed is observed at the beginning of \Cref{fineprojectivesequences}.
Moreover:
\begin{itemize}
\item
$\le^{\mathbb P/R^{\mathbb P}} $ is reflexive by the reflexivity of $\le^{ \mathbb P }$.
\item
If $x\le^{\mathbb P/R^{\mathbb P}} y\le^{\mathbb P/R^{\mathbb P}} z$ with $x\ne y\ne z$, let
\begin{align*}
 & u\in p^{-1}(x), \\
 & v, v'\in p^{-1}(y), \\
 & w\in p^{-1}(z),
\end{align*}
with $u\le v \binR v'\le w$; by \Cref{alleqclassless} it follows that $u\le v'$, so that $u\le w$ and finally $x\le^{\mathbb P/R^{\mathbb P}} z$.
\item
If $x\le^{\mathbb P/R^{\mathbb P}} y\le^{\mathbb P/R^{\mathbb P}} x$, there are
\begin{align*}
 & u, u'\in p^{-1}(x), \\
 & v, v'\in p^{-1}(y),
\end{align*}
with $u\le v \binR v'\le u'$; by \Cref{fineclassesconvex} it follows that $u \binR v$, and finally $x=y$.
\end{itemize}
\end{proof}

\begin{lemma}
	\label{branchesareclopen}
If $B \in \branches(P_{n})$ then $\bigcup_{a\in B} \cappello{a}_{ \fhi_n}$ is a clopen subset of $\quot{ \mathbb P }{R^{ \mathbb P }} $.
\end{lemma}
\begin{proof}
Since for each $a \in B$ the set $\fhi_{n}^{-1}(a)$ is clopen, it follows that $\bigcup_{a\in B} \fhi_{n}^{-1}(a)$ is clopen.
Let $u, v \in \mathbb P$ be such that $u \in \bigcup_{a\in B} \fhi_{n}^{-1}(a)$ and $u \binR^{\mathbb P} v$.
Then $\fhi_{n}(u) \binR^{P_{n}} \fhi_{n}(v)$, so $\fhi_{n}(v) \in B$, that is, $v \in  \bigcup_{a\in B} \fhi_{n}^{-1}(a)$.
It follows that $ \bigcup_{a\in B} \fhi_{n}^{-1}(a)$ is $R^{ \mathbb P }$-invariant, so $\bigcup_{a\in B} \cappello{a}_{ \fhi_n}= p [\bigcup_{a\in B} \fhi_{n}^{-1}(a)]$ is open, thus clopen.
\end{proof}

A converse of the above also holds.

\begin{lemma}
	\label{clopensareunionofbranches}
Let $C$ be a clopen subset of $\quot{ \mathbb P }{R^{ \mathbb P }}$.
There is $n \in \N$ such that for all $m \ge n$, there is $S \subseteq \branches(P_{m})$ for which $C = \bigcup_{a\in\bigcup S} \cappello{a}_{ \fhi_m}$.
\end{lemma}
\begin{proof}
First notice that it is enough to show that there are some $n\in \N $ and $S\subseteq \branches (P_n)$ for which $C=\bigcup_{a\in\bigcup S} \cappello{a}_{ \fhi_n}$.
Indeed, assuming this, let $m\ge n$.
Then $( \fhi_n^m)^{-1}(\bigcup S)=\bigcup T$ for some $T\subseteq \branches (P_m)$, and $C=\bigcup_{a\in\bigcup T} \cappello{a}_{ \fhi_m}$.

Since $p^{-1}(C)$ is compact and open and the sets $ \setnew{ \fhi_n^{-1}(a)}{n\in \N , a\in A_n} $ form a basis for the topology of $ \mathbb P $, there exist $n\in \N $ and a subset $B\subseteq P_n$ such that $p^{-1}(C)=\bigcup_{a\in B} \fhi_n^{-1}(a)$, so that $B= \fhi_n[p^{-1}(C)]$.

We prove that $B=\bigcup S$ for some $S\subseteq \branches (P_n)$.
If this were not the case, there would exist $a, a'\in P_n$ with $a, a'$ consecutive with respect to $\le^{P_n}$ and $a\in B, a'\notin B$; in particular, $a \binR a'$.
If $u, u'\in \mathbb P $ are such that $ \fhi_n(u)=a, \fhi_n(u')=a', u \binR u'$, then $u\in p^{-1}(C), u'\notin p^{-1}(C)$ contradict the fact that $p^{-1}(C)$ is $R^{ \mathbb P }$-invariant.
The proof is concluded by observing that:
\[
C=p(p^{-1}[C])=p[\bigcup_{a\in B} \fhi_n^{-1}(a)]=\bigcup_{a\in B} \cappello{a}_{ \fhi_n}.
\]
\end{proof}

\section{Fences}\label{sec:fences}

\begin{definition} \label{deffence}
A \emph{fence} is a compact metrizable space whose connected components are either points or arcs.
A fence $Y$ is \emph{smooth} if there is a closed partial order $\preceq$ on $Y$ whose restriction to each connected component of $Y$ is a total order.
\end{definition}

We call \emph{arc components} of a fence the connected components which are arcs, and \emph{singleton components} those which are points.
We denote by $\mathrm{E}(Y)$ the set of endpoints of a fence $Y$; equivalently, $ \mathrm{E}(Y)$ is the set of endpoints of the connected components of $Y$.
The \emph{Cantor fence} is the space $\Can \times [0, 1]$; it is a smooth fence, as witnessed by the product of equality on $ \Can $ and the usual ordering of $[0, 1]$: we denote this order by $\trianglelefteq $.

\Cref{smoothembedd} below establishes that smooth fences are, up  to homeomorphism, the compact subspaces of the Cantor fence.
It may be confronted with \cite{MR1020278}*{Proposition 4}, stating that  smooth fans are, up to homeomorphism, the subcontinua of the Cantor fan, which is the fan obtained by identifying in the Cantor fence the set $ \Can \times\{ 0\} $ to a point.

Recall that if $X$ is a topological space and $f: X \to [0, 1]$ is a function, then $f$ is \emph{lower semi-continuous} (l.s.c.) if $ \setnew{x\in X}{f(x)\leq y}$ is  closed for each $y \in [0, 1]$ and is \emph{upper semi-continuous} (u.s.c.) if $ \setnew{x\in X}{f(x)\geq y}$ is closed for each $y \in [0, 1]$.

Let $X$ be a zero-dimensional, compact, metrizable space and $m, M: X \to [0, 1]$ be two functions.
We say that $(\mini , \maxi)$ is a \emph{fancy pair} if
\begin{itemize}
\item
$m$ is l.s.c.;
\item
$M$ is u.s.c.;
\item
$m(x) \le M(x)$, for all $x \in X$.
\end{itemize}

If $(\mini, \maxi)$ is a fancy pair of functions on $X$, let $D_{\mini}^{\maxi} = \setnew{(x, y) \in X \times [0, 1]}{\mini(x) \le y \le \maxi(x)}$.
Then $D_{\mini}^{\maxi}$ is a closed subset of $X \times [0, 1]$.
Indeed, let $(x_{n}, y_{n}) \in D_{\mini}^{\maxi}$, and $(x, y) = \lim (x_{n}, y_{n})$.
Then for each $\varepsilon > 0$, there exists $n \in \N$ such that for all $m>n$,
\[
\mini(x) - \varepsilon < \mini(x_{m}) \le y_{m} \le \maxi(x_{m}) < \maxi(x) + \varepsilon,
\]
so $\mini(x) \le y \le \maxi(x)$, thus $(x, y) \in D_{\mini}^{\maxi}$.

\begin{theorem} \label{smoothembedd}
Let $Y$ be a fence.
Then the following are equivalent:
\begin{enumerate}
\item
$Y$ is a smooth fence.
\item
	\label{itm:stronglycomp}
\sloppy
There exists a closed partial order $\preceq $ on $Y$ whose restriction to each connected component is a total order and such that two elements are $\preceq$-comparable if and only if they belong to the same connected component.
\item There is a continuous injection $f:Y\to \Can \times [0, 1]$.
\item There is a continuous injection $f:Y\to \Can \times [0, 1]$ such that for each $x \in \Can$, the set $f[Y]\cap ( \set{x} \times [0, 1])$ is connected (possibly empty).
\item
There is a closed, non-empty, subset $X$ of $ \Can $ and a fancy pair $(\mini, \maxi)$ of functions on $X$ such that $Y$ is homeomorphic to $D_{\mini}^{\maxi}$.
\end{enumerate}
\end{theorem}

\begin{proof}
The implications $(2)\Rightarrow (1)$ and $(4)\Rightarrow (3)$ are immediate.
The implications $(3)\Rightarrow (1)$ and $(4)\Rightarrow (2)$ follow by copying on $Y$ the restriction of the order $\trianglelefteq $ on the Cantor fence to the image of $Y$ under the embedding.

For $(4)\Rightarrow (5)$, let $X = \pi_{1}[f[Y]]$ be the projection of $f[Y]$ on $\Can$ and, for $x \in X$, let $\mini(x)= \min \setnew{y \in [0, 1]}{ (x, y) \in f[Y]}$ and $\maxi(x)= \max \setnew{y \in [0, 1]}{ (x, y) \in f[Y]}$.
Clearly $\mini(x) \le \maxi(x)$, for all $x \in X$, and $\mini, \maxi$ are l.s.c, u.s.c., respectively, since $f[Y]$ is closed.
Then $(\mini, \maxi)$ is a fancy pair of functions on $X$ and $D_{\mini}^{\maxi} = f[Y]$.

For $(5)\Rightarrow (4)$, suppose that there are a closed, non-empty, subset $X$ of $ \Can $ and a fancy pair $(\mini, \maxi)$ of functions on $X$ such that there is a homeomorphism $f: Y \to D_{\mini}^{\maxi}$.
Then $f$ is the required injection.

It thus remains to establish $(1)\Rightarrow (4)$.
By \cite{Kuratowski1968}*{§46, V, Theorem 3}, there is a continuous map $f_{0} : Y \to \Can$ such that $f_{0}(x)=f_{0}(x')$ if and only if $x, x'$ belong to the same connected component.

By \cite{Carruth1968}, any compact metrizable space with a closed partial order can be embedded continuously and order-preservingly in $[0, 1]^{\N}$ with the product order.
Let $h: Y \to [0, 1]^{\N}$ be such an embedding.
Let $f_{1}: Y \to [0, 1]$ be defined by $f_{1}(x) = d(\mathbf{0}, h(x))$, where $d$ is the product metric on $[0, 1]^{\N}$ and $\mathbf{0} = (0, 0, \dots )$.
Then $f_{1}$ is the composition of two continuous functions, so it is continuous, and its restriction to each connected component of $Y$ is injective, since $d(\mathbf{0}, x) < d(\mathbf{0}, y)$ whenever $x$ is less than $y$ in the product order on  $[0, 1]^{\N}$.

Let $f: Y \to \Can \times [0, 1]$ be defined by $f(x) = (f_{0}(x), f_{1}(x))$.
Then $f$ is the continuous embedding which we were seeking.
\end{proof}

Note that if $\preceq $ is the closed order on $Y$ used for embedding $Y$ into the Cantor fence, the embedding $f$ of $(1)\Rightarrow (4)$ in the preceding proof also embeds  $\preceq $ in $\trianglelefteq $.

For later use, we say that an order relation on the fence $Y$ is \emph{strongly compatible} if it satisfies \eqref{itm:stronglycomp} of \Cref{smoothembedd}.
For example, $\trianglelefteq $ is a strongly compatible order on the Cantor fence.

\begin{remark}
Condition \eqref{itm:stronglycomp} in \Cref{smoothembedd} implies that the ternary relation $T$ on a smooth fence $Y$, defined by $T(x, y, x')$ if and only if $x=y=x'$ or $y$ belongs to the arc with endpoints $x, x'$, is closed.
We do not know if requiring that this relation is closed is equivalent or strictly weaker than the conditions in \Cref{smoothembedd}.
\end{remark}

\subsection{Smooth fences and \texorpdfstring{$\diamonds$}{F0}}

We turn to proving that smooth fences are exactly the spaces which can be approximated by fine projective sequences in $\diamonds$.
One direction is \Cref{thmpidiamondfence}, the other is \Cref{smoothfencespidiamond}.

\begin{theorem} \label{thmpidiamondfence}
Let $(P_n, \fhi_n^m)$ be a fine projective sequence in $\diamonds$, with projective limit $ \mathbb P $ and let $p: \mathbb P \to \quot{ \mathbb P }{R^{ \mathbb P }} $ be the quotient map.
Then $ \quot{ \mathbb P }{R^{ \mathbb P }} $ is a smooth fence.

The connected components of $ \quot{ \mathbb P }{R^{ \mathbb P }} $ are the maximal chains of the order $\le^{ \mathbb P /R^{ \mathbb P }}$.
They are the sets of the form $p[B]$, where $B$ is a maximal chain in $ \mathbb P $; in particular, if $B$ has more than two elements, then $p[B]$ is an arc.
\end{theorem}

\begin{proof}
The relation $\le^{\mathbb P/R^{\mathbb P}} $ on $ \quot{ \mathbb P }{R^{ \mathbb P }} $ is a closed order by \Cref{corollarytwentyfour}.

If $x\not \le^{\mathbb P/R^{\mathbb P}} y\not \le^{\mathbb P/R^{\mathbb P}} x$, pick $u\in p^{-1}(x), v\in p^{-1}(y)$ and let $n\in \N $ be such that $ \fhi_n(u)\nleq \fhi_n(v)\nleq \fhi_n(u)$.
This implies that $ \fhi_n(u), \fhi_n(v)$ belong to distinct maximal chains $B, B'$, respectively, of $P_n$.
By \Cref{branchesareclopen}, $p[ \fhi_n^{-1}(B)], p[ \fhi_n^{-1}(B')]$ are clopen subsets of $ \quot{ \mathbb P }{R^{ \mathbb P }} $ separating $x$ and $y$, so $x, y$ belong to distinct connected components of $ \quot{ \mathbb P }{R^{ \mathbb P }} $.

If $x\le^{\mathbb P/R^{\mathbb P}} y$, let $u, v\in \mathbb P $ with $u\in p^{-1}(x), v\in p^{-1}(y), u\le v$.
Since $[u, v]$ is $R$-connected by \Cref{connectedintervals}, from \Cref{oldparagraph} it follows that $p[[u, v]]$ is a connected subset of $ \quot{ \mathbb P }{R^{ \mathbb P }} $ containing $x, y$.
Therefore $x, y$ belong to the same connected component.

These two facts show that the connected components of $ \quot{ \mathbb P }{R^{ \mathbb P }} $ are the maximal chains of $\le^{ \mathbb P /R^{ \mathbb P }}$ or, equivalently, the sets of the form $p[B]$, where $B$ ranges over the maximal chains of $ \mathbb P $.
If in particular $B$ has more than two points, then $p[B]$ is not a singleton by \Cref{fineclassesconvex}.

Thus it remains to show that the non-singleton connected components of $ \quot{ \mathbb P }{R^{ \mathbb P }} $ are arcs.
So let $K$ be a non-singleton connected component of $ \quot{ \mathbb P }{R^{ \mathbb P }} $.
By the above, the restriction of $\le^{\mathbb P/R^{\mathbb P}} $ to $K$ is a closed total order, so it is complete as an order by \cite{Basso}*{Lemma 15}, and has a minimum and a maximum that are distinct.
Moreover, it is dense as $K$ is connected, so it is a separable order as open intervals are open subsets in the topology of $K$.
Using \cite{Rosens1982}*{Theorem 2.30}, the restriction of $\le^{\mathbb P/R^{\mathbb P}} $ to $K$ is an order of type $1+\lambda +1$, where $\lambda $ is the order type of $ \R $; as the sets of the form $ \setnew{x\in K}{x<^{ \mathbb P /R^{ \mathbb P }}z} $ and $ \setnew{x\in K}{z<^{ \mathbb P /R^{ \mathbb P }}x} $ are open subsets of $K$, this means that there is a continuous bijection $K\to [0, 1]$, which is therefore a homeomorphism.
\end{proof}

The converse of \Cref{thmpidiamondfence} is proved in \Cref{smoothfencespidiamond},
for which we need the following lemma and definition.

\begin{lemma}\label{usualclaimnew}
Let $X$ be a zero-dimensional compact metrizable space and $(\mini, \maxi)$ a fancy pair of functions on $X$.
For each $\varepsilon>0$ and each clopen partition $\mathcal U$ of $\basis$ there is a clopen partition
$\mathcal W $ refining $\mathcal U$, such that for all $U \in \mathcal W $
there is $x_{U} \in U$ such that:
\begin{equation} \label{anotherpropertyone}
\mini (x_U)- \min \setnew{ \mini (x)}{x\in U} <\varepsilon ,\qquad\max \setnew{ \maxi (x)}{x\in U} - \maxi (x_U)<\varepsilon .
\end{equation}
\end{lemma}

\begin{proof}
By dealing with one element of $ \mathcal U $ at a time, it is enough to show that given a zero-dimensional compact metrizable space $X$, a fancy pair $( \mini , \maxi )$, and $\varepsilon >0$, there is a clopen partition $ \mathcal W = \set*{W_0,\ldots ,W_k}$ of $X$ such that for all $U\in \mathcal W $ there is $x_U\in U$ for which \eqref{anotherpropertyone} holds.

For any clopen set $U \subseteq X$, let
\[
m_U=\min \setnew{ \mini (x)}{x\in U} ,\qquad M_U=\max \setnew{ \maxi (x)}{x\in U} .
\]
If there exists $x_X\in X$ satisfying \eqref{anotherpropertyone}, then we are done by letting $k=0,W_0=X$.
Otherwise, let $U_0= \setnew{x\in X}{ \maxi (x)<M_X- \frac{\varepsilon }2 } $.
This is an open set, and since there is no $x_{X}$ satisfying \eqref{anotherpropertyone}, it contains the closed, non-empty, set $C_0= \setnew{x\in X}{ \mini (x)\le m_X+ \frac{\varepsilon }2 } $.
By the zero-dimensionality of $X$ and the compactness of $C_0$, let $V_0$ be clopen such that $C_0\subseteq V_0\subseteq U_0$.
Notice that
\[
m_{V_0}=m_X,\qquad M_{V_0}<M_X- \frac{\varepsilon }2 .
\]
If there exists $x_{V_0}\in V_0$ such that \eqref{anotherpropertyone} holds, then set $W_0=V_0$.
Otherwise repeat the process within $V_0$, to find a clopen set $V_1$ with $C_0\subseteq V_1\subseteq V_0$ and
\[
m_{V_1}=m_{V_0}=m_X,\qquad M_{V_1}<M_{V_0}- \frac{\varepsilon }2 <M_X-\varepsilon .
\]
Thus this process must stop, yielding finally a clopen subset $W_0$ such that $C_0\subseteq W_0\subseteq U_0$ and there exists $x_{W_0}\in W_0$ for which \eqref{anotherpropertyone} holds.

Now start the process over again within $X'=X\setminus W_0$, which is non-empty by case assumption.
Since $C_0\subseteq W_0\subseteq U_0$, it follows that
\[
m_X+ \frac{\varepsilon }2 <m_{X'},\qquad M_{X'}=M_X.
\]

If there exists $x_{X'}\in X'$ satisfying \eqref{anotherpropertyone}, we are done by letting $k=1,W_1=X'$.
Otherwise we eventually produce a clopen subset $W_1$ of $X'$ containing $C_1= \setnew{x\in X'}{ \mini (x)\le m_{X'}+ \frac{\varepsilon }2 } $, contained in $U_1= \setnew{x\in X'}{ \maxi (x)<M_{X'}- \frac{\varepsilon }2 } $, and such that there exists $x_{W_1}\in W_1$ satisfying \eqref{anotherpropertyone}.
Set $X''=X\setminus (W_0\cup W_1)$ and notice that
\[
m_X+\varepsilon <m_{X'}+ \frac{\varepsilon }2 <m_{X''}, \qquad M_{X''}=M_X.
\]
Thus the process eventually stops, providing the desired partition $ \mathcal W $.
\end{proof}

%%%%%%%%%%%%%%%%%%%%%%%%%

\begin{theorem}\label{smoothfencespidiamond}
Let $Y$ be a smooth fence with a strongly compatible order $\preceq$.
Then there exists a fine projective sequence of structures $(P_n, \fhi_n^m)$ from $\diamonds$ approximating $Y$ in such a way that, denoting by $ \mathbb P $ the projective limit:
\begin{itemize}
\item[a)]
the quotient map $p: \mathbb P \to \quot{ \mathbb P }{R^{ \mathbb P }} $ is irreducible;
\item[b)]
there is a homeomorphism $g: \quot{ \mathbb P }{R^{ \mathbb P }} \to Y$ that is also an order isomorphism between $\le^{\mathbb P/R^{\mathbb P}} $ and $\preceq$.
\item[c)]
for each $n \in \N$, $a, a' \in P_{n}$, it holds that $a \le^{P_{n}} a'$ if and only if there are $x \in \interior(\cappello{a}_{\fhi_{n}}), x' \in \interior(\cappello{a'}_{\fhi_{n}})$, $g(x)\preceq g(x')$.
\end{itemize}
\end{theorem}

\begin{proof}
By \Cref{smoothembedd} and the remark following it, we can assume that $Y=D_{\mini}^{\maxi}$ for a closed, non-empty $\basis \subseteq \Can $ and a fancy pair $(\mini, \maxi)$ of functions on $\basis$, such that $\preceq$ coincides with the product order $\trianglelefteq$ on $\basis \times [0, 1]$.
We can furthermore assume that $\mini(x) >0, \maxi(x)<1$ for all $x \in \basis$.
Let $\dist$ be the product metric on $\basis \times [0, 1]$.

We first define a homeomorphic copy $Y' = D_{\mini'}^{\maxi'}$ of $Y$ in $X \times (0, 1)$ and a sequence $(\mathcal{U}_{n})_{n \in \N}$ of partition of $\basis$ such that for any $n \in \N$ and $U \in \mathcal U_{n}$, there is $x_{U} \in U$ such that:
\[
\mini'(x_{U}) = \min\setnew{ \mini'(x)}{x\in U}, \qquad
\maxi'(x_{U}) = \max\setnew{ \maxi'(x)}{x\in U}.
\]
This allows us to find a sequence of coverings of $Y'$ which in turn give rise to the $P_{n}$'s.

Let $\mathcal U_{0} = \set{\basis}$ be the trivial clopen partition of $\basis$ and $\beta_{0}: \basis \times \inter \to \basis \times \inter$ be the identity.
Suppose one has defined a clopen partition $\mathcal U_{n}$ of $\basis$ and a homeomorphism $\beta_{n}: \basis \times \inter \to \basis \times \inter$.
Let $\mini^{n}, \maxi^{n}$ be such that $D_{\mini^{n}}^{\maxi^{n}} = \beta_{n}[Y]$.
For any clopen set $U \subseteq X$, denote
\[
\mini_U^n=\min_{x \in U} \mini^{n}(x),\qquad \maxi_U^n= \max_{x \in U} \maxi^{n}(x).
\]
Let $\mathcal U_{n+1}$ refine $\mathcal U_{n}$, have mesh less than $ \frac 1{n+1} $, and satisfy \Cref{usualclaimnew} for $\beta_{n}[Y]$ and $\varepsilon = \nicefrac{1}{2^{n+1}}$.
For each $U \in \mathcal U_{n+1}$ fix $x_{U}$ given by \Cref{usualclaimnew}, additionally we can ask that if $\mini^{n}\restr{U} \ne  \maxi^{n}\restr{U}$, then $\mini^{n}(x_U) < \maxi^{n}(x_U)$.

For any $\ell \in \N$ and any two increasing sequences of real numbers $0< a_{0} < \cdots < a_{\ell -1}< 1$ and $0< b_{0} < \cdots < b_{\ell -1}< 1$, let $\pwlin{ \vec a}{\vec b} : \inter \to \inter$ be the piecewise linear function mapping $0 \mapsto 0, 1 \mapsto 1, a_{i} \mapsto b_{i}$ for each $i < \ell$:
\[
\pwlin{ \vec a}{\vec b}(y)=
\begin{cases}
	\dfrac{b_{0}}{a_{0}}y & \text{ if } y \le a_{0}, \\
	\dfrac{b_{i+1} - b_i}{a_{i+1} - a_i}y + \dfrac{b_{i}a_{i+1}-a_{i}b_{i+1}}{a_{i+1} - a_i} & \text{ if } a_{i} < y \le a_{i+1},  i <\ell -1, \\
	\dfrac{1-b_{\ell-1}}{1-a_{\ell-1}}y + \dfrac{b_{\ell-1} - a_{\ell-1}}{1-a_{\ell-1}} & \text{ if }  y > a_{\ell-1}.
\end{cases}
\]
Note that, for fixed $\ell $, this is a continuous function of the variables $a_0,\ldots ,a_{\ell -1},y$.

If for each $x \in U$, $\mini^{n}(x) = \maxi^{n}(x)$, then $\mini^{n} \restr{U} = \maxi^{n} \restr{U}: U \to [0, 1]$ is a continuous function, as it is both l.s.c.\ and u.s.c..
If follows that if we fix $x_{U} \in U$ and define $\alpha_{U}: U \times \inter \to U \times \inter$ as $\alpha_{U} (x, y) = \left(x,  \pwlin{\mini^{n}(x)}{\mini^{n}(x_{U})}(y)\right)$, then $\alpha_{U}$ is a homeomorphism.
Notice that, in this case, $\alpha_U$ sends $\beta_n[Y]\cap (U\times [0,1])$ onto $U\times \set*{ \mini^n(x_U)} $; in particular, if $\beta_n[Y]\cap (U\times [0,1])=U\times \set*{ \mini^n(x_U)} $, then $\alpha_U$ is the identity.

If, on the other hand, $x_{U} \in U$ is such that $\mini^{n}(x_{U}) < \maxi^{n}(x_{U})$, we define the functions $f_{U}, g_{U}, f'_{U}, g'_{U}: U \to (0, 1)$ as follows:
\begin{align*}
f_{U}(x) &= \begin{cases}
	\mini^{n}_{U} & \text{ if } x \ne x_{U} \\
	\mini^{n}(x_{U}) & \text{ if } x = x_{U}
	\end{cases} \\
g_{U}(x) &= \min \set*{\mini^{n}(x), \mini^{n}(x_{U}) }\\
f'_{U}(x) &= \begin{cases}
	\maxi^{n}_{U} & \text{ if } x \ne x_{U} \\
	\maxi^{n}(x_{U}) & \text{ if } x = x_{U}
	\end{cases} \\
g'_{U}(x) &= \max \set*{\maxi^{n}(x), \maxi^{n}(x_{U}) }
\end{align*}

It is immediate by their definitions that $f_{U}, g'_{U}$ are u.s.c., $g_{U}, f'_{U}$ are l.s.c., and that:
\[
\mini^{n}_{U} \le f_{U} \le g_{U} \le \mini^{n}(x_{U}) < \maxi^{n}(x_{U}) \le g'_{U} \le f'_{U} \le \maxi^{n}_{U}.
\]
By the Katětov–Tong insertion theorem there are $h_{U}, h'_{U}: U \to (0, 1)$ continuous, such that $f_{U} \le h_{U} \le g_{U}$ and $g'_{U} \le h'_{U} \le f'_{U}$.

We define $\alpha_{U}: U \times \inter \to U \times \inter$ to be:
\[
\alpha_{U}(x, y) = \left( x, \pwlin{h_{U}(x) , h'_{U}(x)}{\mini^{n}_{U}, \maxi^{n}_{U}}(y) \right).
\]
Then $\alpha_{U}$ is a homeomorphism.

Define $\alpha_{n} = \bigsqcup_{U \in \mathcal U_{n+1}} \alpha_{U}$, so  $\alpha_{n} \in \homeo(\basis \times \inter)$.
Finally let $\beta_{n+1} = \alpha_{n} \beta_{n}$ and $ \mini^{n+1}, \maxi^{n+1}$ be such that $\beta_{n+1}[Y]=D_{ \mini^{n+1}}^{ \maxi^{n+1} }$.
Notice that for any $U \in \mathcal U_{n+1}$
\begin{equation} \label{boundsarethesame}
\mini^{n+1}(x_{U}) = \mini_U^n= \mini^{n+1}_{U}\quad \text{and} \quad \maxi^{n+1}(x_{U}) = \maxi_U^n= \maxi^{n+1}_{U}.
\end{equation}
Let $(x, y), (x, y') \in \beta_{n}[Y]$, and suppose that $x \in U \in \mathcal U_{n+1}$, $y \le y'$.
Then $\mini^{n}_{U} \le  h_{U}(x)  \le y \le y' \le h'_{U}(x) \le \maxi^{n}_{U}$ so:
\[
\pwlin{h_{U}(x) , h'_{U}(x)}{\mini^{n}_{U}, \maxi^{n}_{U}}(y') -\pwlin{h_{U}(x) , h'_{U}(x)}{\mini^{n}_{U}, \maxi^{n}_{U}}(y) =
\dfrac{\maxi_{U}^{n} - \mini_{U}^{n}}{h'_{U}(x) - h_{U}(x)} (y'-y) \ge y' - y,
\]
that is, $\dist( (x, y), (x, y') ) \le \dist( \alpha_{U}(x, y),  \alpha_{U}(x, y'))$.
It follows that for $(x, y), (x, y') \in Y$:
\begin{equation}
	\label{eq:noncontracting}
\dist( (x, y), (x, y') ) \le \dist( \beta_{n+1}(x, y),  \beta_{n+1}(x, y')).
\end{equation}

We prove that the sequence $(\beta_{n})_{n \in \N}$ is Cauchy with respect to the supremum metric $\dist_{\sup}$.
Indeed, for each $n$, $\dist_{\sup}(\mathrm{id}, \alpha_{n}) < \nicefrac{1}{2^{n+1}}$ by the definition of the points $x_U$.
By right invariance of the supremum metric and the triangle inequality, whenever $n<m$,
\begin{multline*}
\dist_{\sup}( \beta_{n}, \beta_{m})= \dist_{\sup}( \beta_{n}, \alpha_{m-1} \cdots \alpha_{n} \beta_{n}) = \dist_{\sup}(\mathrm{id}, \alpha_{m-1} \cdots \alpha_{n})\le \\
\le \dist_{\sup}(\mathrm{id}, \alpha_{m-1} ) + \cdots + \dist_{\sup}(\mathrm{id}, \alpha_{n}) < \sum_{i = n+1}^{m} \nicefrac{1}{2^{i}} < \nicefrac{1}{2^{n}}.
\end{multline*}
It follows that for each $\varepsilon$, there is $n$ such that for each $m>n$, $\dist_{\sup}(\beta_{n}, \beta_{m}) < \varepsilon$.

Since the space of continuous functions from $X \times \inter$ in itself with the supremum metric is complete, the sequence $(\beta_{n})_{n \in \N}$ has a limit, which we denote by $\beta$.
Since it is the limit of surjective functions, $\beta$ is surjective.
We prove that it is injective on $Y$, that is, that its restriction to $Y$ is a homeomorphism onto $Y'=\beta [Y]$.

Let $(x, y), (x', y') \in Y $.
If $x \neq x'$, then $\beta(x, y) \neq \beta(x', y')$ as $\beta$ is the identity on the first coordinate.
So suppose $x=x'$.
Since \eqref{eq:noncontracting} holds for each $n \in \N$, we have that $\dist( (x, y), (x, y') ) \le \dist( \beta(x, y),  \beta(x, y'))$, so $\beta$ is injective on $Y$.

By \eqref{boundsarethesame} it follows that $Y'\subseteq X\times [ \mini_X, \maxi_X]\subseteq X\times (0,1)$.
Notice that $x \trianglelefteq x'$ if and only if $\beta(x) \trianglelefteq \beta(x')$.
Let $\mini', \maxi'$ be such that $D_{\mini'}^{\maxi'} = Y'$.
As above, for any clopen set $U \subseteq X$, denote $\mini'_U=\min_{x \in U} \mini'(x)$ and $\maxi'_U= \max_{x \in U} \maxi'(x)$.
For any $n \in \N$ and $U \in \mathcal U_{n+1}$, $\mini'(x_{U}) = \mini'_{U}$ and $\maxi'(x_{U}) = \maxi'_{U}$.
This is clear if $\mini^{n} \restr{U} = \maxi^{n} \restr{U}$.
Otherwise, we have seen that $\mini^{n+1}(x_{U}) = \mini^{n+1}_{U}$.
Assume that $ \mini^r(x_U)= \mini_U^r$ for some $r\ge n+1$.
Given any $U'\in \mathcal U_{r+1}$ with $U'\subseteq U$, by \eqref{boundsarethesame} it follows that $ \mini^r(x_U)\le \mini_{U'}^r= \mini_{U'}^{r+1}$, whence $ \mini^r(x_U)= \mini_U^r= \mini_U^{r+1}$ and, in particular, $\forall r\ge n+1$ we have $ \mini^r(x_U)= \mini^{n+1}(x_U)= \mini_U^r$, which allows to conclude $ \mini'(x_U)= \mini^{n+1}(x_U)= \mini'_U$.
Similarly, $ \maxi'(x_U)= \maxi'_U$.

Let
$
K_{U} = \setnew*{(x_{U}, y)}{\mini'_{U} \le y \le \maxi'_{U}} =  (\set{x_{U}} \times \inter) \cap Y'.
$

Let $x_{0} = 0, x_{1}=1$.
Let $\Theta = \setnew*{x_{\nicefrac{m}{2^{n}}} }{n\ge 1, 1\le m<2^{n} }$ be a countable dense subset of $(0, 1)\setminus \setnew{\mini_{U}, \maxi_{U}}{U \in \mathcal U_{n}, n \in \N}$, indexed in such a way that $x_{p} < x_{q}$ if and only if $p<q$.

For $n\ge 0$, let:
\[
\mathcal I_{n} = \setnew*{\left[x_{\nicefrac{m}{2^{n}}}, x_{\nicefrac{(m+1)}{2^{n}}}\right]}{0\leq m\le 2^{n}-1} .
\]
Then define:
\[
 \mathcal C_{n} = \setnew*{U \times I}{U \in \mathcal U_{n}, I \in \mathcal I_{n}}.
\]

Notice that for each $n$:
\begin{enumerate}
\item $ \mathcal C_n$ is a regular quasi-partition of $ X \times [0, 1]$,
\item \label{uniquerefinement} $\forall C\in \mathcal C_{n+1}\ \exists !C'\in \mathcal C_n\ C\subseteq C'$.
\end{enumerate}
The mesh of $\mathcal C_{n}$ tends to $0$ as $n$ grows, since $\Theta $ is dense and the mesh of $ \mathcal U_n$ goes to $0$.
Endow each $\mathcal C_{n}$ with the discrete topology and give $\mathcal C_{n}$ an $\mathcal L_R$-structure by letting
\begin{itemize}
\item
$C \binR^{\mathcal C_{n}} C'$ if and only if $C \cap C' \neq \emptyset$,
\item
$C \le^{\mathcal C_{n}} C'$ if and only if there are $x \in \interior(C), x' \in \interior(C')$ with $x\trianglelefteq x'$.
\end{itemize}
Then $\mathcal C_{n} \in \diamonds$.
Notice that $C, C'$ are $\le^{\mathcal C_{n}}$-comparable if and only if $\pi_{1}[C] = \pi_{1}[C']$, where $\pi_1$ is the projection onto $X$.

For each $n$, define
\[
P_{n} = \setnew*{C \in \mathcal C_{n}}{C\cap Y' \neq \emptyset}
\]
and have it inherit the $\mathcal L_{R}$-structure of $\mathcal C_{n}$.

\begin{claim}
$P_n= \setnew*{C\in \mathcal C_n}{C\cap K_{\pi_1[C]}\ne\emptyset } $.
\end{claim}
\begin{proof}
If $C\in \mathcal C_n$ is such that $C\cap Y'\ne\emptyset $, let $(x,y)\in C\cap Y'$.
As $\mini'(x_{\pi_1[C]}) = \mini'_{\pi_1[C]}$ and $\maxi'(x_{\pi_1[C]}) = \maxi'_{\pi_1[C]}$, it follows that $(x_{\pi_1[C]},y)\in C\cap K_{\pi_1[C]}$.
\end{proof}

If $U\in \mathcal U_n$, the projections of endpoints of $K_{U}$ on the second coordinate
do not belong to $\Theta $.
This implies that if $C\cap Y'\ne\emptyset $, then actually $ \interior (C)\cap K_{\pi_1[C]}\ne\emptyset $.

\begin{claim}
	\label{cl:inFzero}
$P_{n} \in \diamonds$ and $C \le^{P_{n}} C'$ if and only if there are $x \in \interior (C)\cap K_{\pi_1[C]}, x' \in \interior (C')\cap K_{\pi_1[C']}$, such that $x \trianglelefteq x'$.
\end{claim}
\begin{proof}[Proof of the claim]
Let $C, C' \in P_{n}$, they are $\le^{P_{n}}$-comparable if and only if $U = \pi_{1}[C]=\pi_{1}[C']$, so if and only if $C\cap K_{U} \neq \emptyset$, $C'\cap K_{U} \neq \emptyset$, if and only if $ \interior (C)\cap K_U\ne\emptyset , \interior (C')\cap K_U\ne\emptyset $.
In particular $C \le^{P_{n}} C'$ if and only if there are $x \in \interior (C)\cap K_{U}, x'\in \interior (C')\cap K_{U}$, with $x \trianglelefteq x'$.

So suppose $C, C' \in P_{n}$ and $D \in \mathcal C_{n}$ with $C \le^{\mathcal C_{n}} D \le^{\mathcal C_{n}} C'$.
Then $K_{\pi_1[D]}\cap D\ne\emptyset $, so $D\in P_n$.
Therefore $P_{n}$ is a $\le^{ \mathcal C_n}$-convex substructure of $\mathcal C_{n}$, so $P_{n} \in \diamonds$.
\end{proof}

For each $n \in \N$ and $m\ge n$, let $\fhi^{m}_{n}:P_{m} \to P_{n}$ be the inclusion map, that is $\fhi^{m}_{n}(C) = D$ if and only if $C \subseteq D$.
Notice that this is well defined as $\forall C\in \mathcal C_{m}\ \exists ! D\in \mathcal C_n\ C\subseteq D$ and
\[
C \in P_{m}\Rightarrow C\cap Y' \neq \emptyset\Rightarrow D\cap Y' \neq \emptyset\Rightarrow D\in P_n.
\]
Clearly $\fhi^{m}_{n} = \fhi^{n+1}_{n} \cdots \fhi^{m}_{m-1}$ for $n<m$.

\begin{claim}
Each $\fhi^{m}_{n}$ is an epimorphism.
\end{claim}
\begin{proof}[Proof of the claim]
We prove that $\fhi^{m}_{n}$ is $\mL_{R}$-preserving.
Indeed, notice that $C \cap C' \neq \emptyset$ implies that $\fhi^{m}_{n}(C) \cap \fhi^{m}_{n}(C') \neq \emptyset$, so $C \binR^{P_{m}} C'$ implies $\fhi^{m}_{n}(C) \binR^{P_{n}} \fhi^{m}_{n}(C')$.
Moreover, if $x \in \interior(C) \cap K_U$ then $x \in \interior(\fhi^{m}_{n}(C))\cap K_U$, so $C \le^{P_{m}} C'$ implies $\fhi^{m}_{n}(C) \le^{P_{n}} \fhi^{m}_{n}(C')$.

Let $B \in \branches(P_{n})$ and let $U \in \mathcal U_{n}$ be such that $C\cap K_U\ne\emptyset $ for every $C \in B$.
Let $B' \in \branches(P_{m})$ be such that $K_{U} \subseteq \bigcup B'$.
Then $\fhi^{m}_{n}[B']=B$.
We conclude by \Cref{iffchain}.
\end{proof}

We have thus established that  $(P_{n}, \fhi_{n}^m)$ is a projective sequence.
Let $ \mathbb P $ denote its projective limit.

\begin{claim}
	\label{isfine}
The projective sequence $ (P_{n}, \fhi_{n}^m) $ is fine.
\end{claim}

\begin{proof}[Proof of the claim]
Relation $R^{ \mathbb P }$ is reflexive and symmetric, since all $R^{P_n}$ are.

To conclude use \Cref{lemmafour}, the fact that the mesh of $(P_n)$ goes to $0$, and the fact that elements of $P_n$ are $R^{P_n}$-related if and only if their distance is $0$.
\end{proof}

Then $ \quot{ \mathbb P }{R^{ \mathbb P }}$ is homeomorphic to $Y'$.
Indeed, let $f: \mathbb P \to Y'$ be the continuous
map defined by letting $f( (C_{n})_{n\in \N } )$ be the unique element of $ \bigcap_{n \in \N } C_n$.
Notice that $f$ is well defined since the mesh of the $P_{n}$'s goes to $0$, and $\bigcap_{n \in \N } C_n \subseteq Y'$ as $C_{n} \cap Y' \neq \emptyset$, for each $n$, and $Y'$ is closed.
Moreover $f$ is surjective, since each $P_{n}$ is a covering of $Y'$.
Also $f((C_{n})_{n\in \N } ) = f((C'_{n})_{n\in \N } )$ if and only if $\bigcap_{n \in \N } C_n = \bigcap_{n \in \N } C'_{n}$ if and only if $C_{n} \binR^{P_{n}} C'_{n}$ for each $n$, if and only if $(C_{n})_{n\in \N } \binR^{ \mathbb P } (C'_{n})_{n\in \N }$, so $f$ induces a homeomorphism $g': \quot{ \mathbb P }{R^{ \mathbb P }}\to Y'$.
Then $g = \beta^{-1} g': \quot{ \mathbb P }{R^{ \mathbb P }} \to Y$ is the desired homeomorphism.

Finally, we prove the statements a), b), and c).

a)
To apply \Cref{singletonsdenseiffirreducible}, it is enough to prove that for every $n\in \N , D\in P_n$, the set $ \fhi_n^{-1}(D)$ contains a point whose $R^{ \mathbb P }$-equivalence class is a singleton.
Since $Q=\bigcap_{m\in \N} \bigcup_{C\in P_m}( \interior(C)\cap Y')$ is dense in $Y'$, let $x\in Q\cap \interior (D)$.
Then for each $m$ there is exactly one $C_m\in P_m$ to which $x$ belongs, so $f^{-1}(x) = \set{ (C_m)_{m\in \N}}$ and the point $(C_m)_{m\in \N }$ is not $R^{ \mathbb P }$-related to any other point; moreover $(C_m)_{m\in \N }\in \fhi_n^{-1}(D)$.

b) We prove that function $g$ defined above is an isomorphism of the orders $\le^{\mathbb P/R^{\mathbb P}} , \preceq$.

Let $x, y\in \quot{ \mathbb P }{R^{ \mathbb R }} $ be distinct and such that $x\le^{\mathbb P/R^{\mathbb P}} y$.
Let $u\in p^{-1}(x), v\in p^{-1}(y)$.
Then $u, v$ are distinct and $u\le v$.
Moreover $\bigcap_{n\in \N } \fhi_n(u)= \set{g(x)}$, $\bigcap_{n\in \N } \fhi_n(v)= \set{g(y)} $.
By the definition of $ \fhi_n(u)\le^{P_n} \fhi_n(v)$ it follows that there exist $w_n\in \interior ( \fhi_n(u)), z_n\in \interior (\fhi_n(v))$ such that $w_n\preceq z_n$.
Since $\lim_{n\rightarrow\infty }w_n=g(x)$, $\lim_{n\rightarrow\infty }z_n=g(y)$, we conclude $g(x)\preceq g(y)$.

If $x, y\in \quot{ \mathbb P }{R^{ \mathbb R }} $ are $\le^{\mathbb P/R^{\mathbb P}} $-incomparable, if $u\in p^{-1}(x), v\in p^{-1}(y)$ it follows that $u, v$ are $\le^{ \mathbb P }$-incomparable.
Consequently, there exists $n\in \N $ such that $ \fhi_n(u), \fhi_n(v)$ are $\le^{P_n}$-incomparable, implying that $g(x), g(y)$ are $\preceq$-incomparable.

c)
This follows by point b) and \Cref{cl:inFzero}.

\end{proof}

As mentioned in the introduction, in \cite{Bartos2015} the Lelek fan is obtained as a quotient of the projective Fraïssé limit of a subclass of $\forests$.
In particular, the Lelek fan is approximable by a fine projective sequence from $\forests$.
We therefore raise the following question, an answer to which would involve proving analogs of \Cref{thmpidiamondfence,smoothfencespidiamond} for $\forests$.
\begin{question}
What is the class of spaces which are approximable by fine projective sequences from $\forests$?
\end{question}

\subsection{Spaces of endpoints of smooth fences}

Given a smooth fence $Y$ and a strongly compatible order $\preceq$ on $Y$, let $ \mathfrak L_{\preceq} (Y), \mathfrak U_{\preceq} (Y)$ be the space of $\preceq $-minimal points of $Y$ and the space of $\preceq $-maximal points of $Y$, respectively.
By the definition of a strongly compatible order, in these sets are contained all endpoints of $Y$:
\[
\mathrm{E} (Y)= \mathfrak L_{\preceq} (Y) \cup \mathfrak U_{\preceq} (Y).
\]
Notice that $x \in \mathfrak L_{\preceq} (Y) \cap \mathfrak U_{\preceq} (Y)$ if and only if $\set{x}$ is a connected component of $Y$.
When the order $\preceq$ is clear from context we suppress the mention of it in $ \mathfrak L_{\preceq} (Y)$ and $\mathfrak U_{\preceq} (Y)$.

\begin{remark}
	\label{rmk:endpointsgraphs}
By \Cref{smoothembedd}, $Y$ is homeomorphic to $D^{\maxi}_{\mini}$ for some fancy pair $(\mini, \maxi)$ of functions with domain a closed subset of $\Can$.
It follows that $\mathfrak L_{\preceq} (Y), \mathfrak U_{\preceq} (Y)$ are homeomorphic to the graphs of $\mini, \maxi$, respectively.
\end{remark}

In this subsection we establish some topological properties of spaces of endpoints of smooth fences.
In particular, we concentrate on the spaces $\mathfrak L_{\preceq} (Y)$, $\mathfrak U_{\preceq} (Y)$, $\mathfrak L_{\preceq} (Y) \cap \mathfrak U_{\preceq} (Y)$.
We therefore fix a smooth fence $Y$ and a strongly compatible order $\preceq$.
By \Cref{smoothfencespidiamond} we can assume that $Y= \quot{ \mathbb P }{R^{ \mathbb P }}$ for some fine projective sequence $(P_n, \fhi_n^m)$ in $\diamonds$ with projective limit $ \mathbb P $, and that $\preceq$ is $\le^{ \mathbb P /R^{ \mathbb P }}$. Let $p: \mathbb P \to \quot{ \mathbb P }{R^{ \mathbb P }} $ be the quotient map.

\begin{lemma}\label{endpointprop}
A point $u \in \mathbb P $ is $\le^{ \mathbb P }$-maximal if and only if for each $n \in \N$ there exists $m>n$ such that $ \fhi_{n}^{m}\left(\max \setnew{a \in P_{m}}{ \fhi_{m}(u) \le a} \right)= \fhi_{n}(u)$.
Analogously, $u \in \mathbb P $ is $\le^{ \mathbb P }$-minimal if and only if for each $n \in \N$ there exists $m>n$ such that $ \fhi_{n}^{m}\left(\min \setnew{a \in P_{m}}{a \le \fhi_{m}(u)} \right)= \fhi_{n}(u)$.
\end{lemma}
\begin{proof}
Suppose $u$ is $\le^{ \mathbb P }$-maximal and fix $n \in \N$.
For $m>n$, let $b_{m} =\max \setnew{a \in P_{m}}{ \fhi_{m}(u) \le a}$.
If for every $m>n$ it holds that $ \fhi_{n}^{m}(b_m)> \fhi_{n}(u)$, let $v_{m} \in \fhi_{m}^{-1}(b_{m}), u_{m} \in \fhi_{m}^{-1}( \fhi_{m}(u))$ be such that $u_{m} \le v_{m}$.
A subsequence $v_{m_k}$ converges to some $v$.
It follows that $u \le^{ \mathbb P } v$, as $u=\lim_{m\rightarrow\infty }u_{m}$ and the order is closed, and $u \neq v$ as $ \fhi_n(v_{m})\ne \fhi_{n}(u)$, for any $m > n$, a contradiction with the maximality of $u$.

Conversely, let $u \in \mathbb P $ be such that for each $n \in \N$ there exists $m>n$ such that $ \fhi_{n}^{m}\left(\max \setnew{a \in P_{m}}{ \fhi_{m}(u) \le a} \right)= \fhi_{n}(u)$ and let $u \le^{ \mathbb P } v$.
Fix $n$, with the objective of showing $ \fhi_{n}(u) = \fhi_{n}(v)$.
Let $m>n$ satisfy the hypothesis;
notice that it implies that $ \fhi_{n}^{m}[\setnew{a \in P_{m}}{ \fhi_{m}(u) \le a}] =\{ \fhi_{n}(u)\} $.
From $u \le v$ it follows that $ \fhi_{m}(u) \le \fhi_{m}(v)$ so $ \fhi_{n}(v) = \fhi^{m}_{n} \fhi_{m}(v) =  \fhi_{n}(u)$.

The case of $u$ $\le^{ \mathbb P }$-minimal is symmetrical.
\end{proof}

\begin{corollary}\label{upperlower}
Given $x \in \mathfrak U \left (
Y
\right ) $ and any open neighborhood $O$ of $x$ in $ Y $, for $m$ big enough the following holds: if $B_m\in \branches(P_m)$ is such that $x\in\bigcup_{a\in B_m} \cappello{a}_{ \fhi_m}$, then $\cappello{\max B_m}_{ \fhi_m}\subseteq O$.
Consequently, $\lim_{m\rightarrow\infty } \cappello{\max B_m}_{ \fhi_m}= \set{x}$.

The same holds for $x \in \mathfrak L \left (
Y
\right )$, upon changing $\max$ to $\min$.
\end{corollary}
\begin{proof}
Let $u = \max p^{-1}(x)$ and $n \in \N$ be such that $ \cappello{ \fhi_{n}(u)}_{ \fhi_n}\subseteq O$.
By \Cref{endpointprop} there is $m>n$ such that $ \fhi_{n}^{m}(\max B_m)= \fhi_{n}(u)$, for $B_m\in \branches(P_{m})$ with $ \fhi_{m}(u) \in B_m$.
This implies that for all $m'\ge m$ if $B_{m'}\in \branches(P_{m'})$ is such that $ \fhi_{m'}(u)\in B_{m'}$ then $ \fhi_n^{m'}(\max B_{m'})= \fhi_n(u)$.
It follows that eventually $\cappello{\max B_m}_{ \fhi_m}\subseteq \cappello{ \fhi_{n}(u)}_{ \fhi_n}\subseteq O$.
\end{proof}

\begin{corollary}\label{branchforcomponent}
For any connected component $K \subseteq Y$
and any open neighborhood $O$ of $K$ in $ Y $, there are $m \in \N , B \in \branches(P_{m})$ such that $ K \subseteq \bigcup_{a\in B} \cappello{a}_{ \fhi_m}\subseteq O$.
\end{corollary}
\begin{proof}
It can be assumed that $O\ne Y$.
Fix a compatible metric on $Y$ and let $\delta$ be the distance between $K$ and  $Y \setminus O$.
Let $u = \min p^{-1}(K), v = \max p^{-1}(K)$ and $n \in \N$ be such that the mesh of $ \cappello{P_n}_{ \fhi_n}$ is less than $\delta$, so that if $a \in P_{n}$ is such that $\cappello{a}_{ \fhi_n}\cap K \neq \emptyset$, then $\cappello{a}_{ \fhi_n}\subseteq O$.
By \Cref{endpointprop} there are $m'>n$ and $B' \in \branches(P_{m'})$ with $ \fhi_{m'}(u) \in B'$ and $ \fhi_{n}^{m'}(\min B')= \fhi_{n}(u)$.
By a second application of  \Cref{endpointprop}, there are $m>m', B \in \branches(P_{m})$ such that $ \fhi_{m}(v) \in B, \fhi_{m'}^{m}(\max B)= \fhi_{m'}(v)$, so $ \fhi_{n}^{m}(\max B)= \fhi_{n}(v)$.
Since $ \fhi^{m}_{m'}(\min B) \ge \min B'$, it follows that $ \fhi^{m}_{n}(\min B) \ge \fhi^{m'}_{n}(\min B') = \fhi_{n}(u)$ by virtue of $ \fhi^{m'}_{n}$ being an epimorphism.
If $a \in B$, then $\fhi_{n}(u) \le \fhi^{m}_{n}(a) \le \fhi_{n}(v)$, so $\cappello{\fhi^{m}_{n}(a) }_{ \fhi_n}\cap K \neq \emptyset$, hence $\cappello{a}_{ \fhi_m}\subseteq \cappello{\fhi^{m}_{n}(a)}_{ \fhi_n} \subseteq O$.
It follows that $\bigcup_{a\in B} \cappello{a}_{ \fhi_m}\subseteq O$.
\end{proof}

\begin{proposition}\label{singlezerodimensional}
Each point of $\mathfrak L (Y) \cap \mathfrak U (Y)$ has a basis of neighborhoods in $Y$ consisting of clopen sets.
In particular, the space $ \mathfrak L (Y)\cap \mathfrak U (Y)$ is zero-dimensional.
\end{proposition}
\begin{proof}

Let $x \in \mathfrak L (Y) \cap \mathfrak U (Y)$ and $O$ be an open neighborhood of $x$ in $Y$.
By \Cref{branchforcomponent} there exist $n \in \N$ and $B \in \branches(P_{n})$ such that $x \in \bigcup_{a\in B} \cappello{a}_{ \fhi_n}\subseteq O$.
By \Cref{branchesareclopen}, $\bigcup_{a\in B} \cappello{a}_{ \fhi_n}$ is clopen in $Y$ and so its trace in $\mathfrak L (Y) \cap \mathfrak U (Y)$ is clopen in $\mathfrak L (Y) \cap \mathfrak U (Y)$.
\end{proof}

Since $\mathfrak L (Y)$ and $\mathfrak U (Y)$ are homeomorphic to graphs of semi-continuous functions with a zero-dimensional domain, by \cite{Dijkstra2010}*{Remark 4.2} we have the following:

\begin{proposition} \label{propluazd}
The spaces $\mathfrak L (Y)$ and $\mathfrak U (Y)$ are almost zero-dimensional.
\end{proposition}

\begin{lemma}\label{maxarepolish}
The spaces $\mathfrak L (Y), \mathfrak U (Y)$ are Polish.
\end{lemma}
\begin{proof}
The set $\mathfrak U(Y) = \setnew{x \in Y}{\forall y \in Y, y \preceq x \lor (x \not \preceq y \land y \not \preceq x)}$ is the co-projection of $\setnew{(x, y)}{y \preceq x \lor (x \not \preceq y \land y \not \preceq x)}$, which is the union of a closed set and an open set of $Y^{2}$, since $\preceq$ is closed.
A union of a closed set and an open set is $G_{\delta}$ and since $Y$ is compact, the co-projection of an open set is open.
Finally, as co-projection and intersection commute, the co-projection of a $G_{\delta}$ is $G_{\delta}$.
We conclude that $\mathfrak U(Y)$ is a $G_{\delta}$ subset of $Y$, thus is Polish.

Similarly for $ \mathfrak L (Y)$.
\end{proof}

\begin{corollary}
\label{singlepolish}
The spaces $ \mathrm{E} (Y)$ and $\mathfrak L(Y) \cap \mathfrak U(Y)$ are Polish.
\end{corollary}

\begin{remark}
	\label{stronglysigmacomplete}
The spaces $\mathfrak L(Y) \setminus \mathfrak U(Y)$ and $\mathfrak U(Y) \setminus \mathfrak L(Y)$ are strongly $\sigma$-complete spaces (that is, they are union of countably many closed and completely metrizable subspaces), since they are $F_{\sigma }$ subsets of a Polish space.
\end{remark}

\section{The \F fence}
	\label{sec:theprojlimit}

We denote by $\prespace$ the projective \F limit of $ \forests $.
Recall from \Cref{pfforests} that $ \diamonds $ is a projective \F family, with the same projective \F limit as $ \forests $.
Therefore, we fix a fundamental sequence $(\fus_n, \fue_n^m)$ in $\diamonds$, with $F_0$ consisting of a single element.

\begin{proposition} \label{lemusualtwo}
The sequence $(\fus_n, \fue_n^m)$ is fine and the quotient map $p: \prespace \to \quot{ \prespace }{R^{ \prespace }}$ is irreducible.
\end{proposition}

\begin{proof}
Let $a, b \in \fus_{n}$ have $R^{\fus_{n}}$-distance $2$.
Say, without loss of generality, $a \binR^{\fus_{n}} c \binR^{\fus_{n}} b$ and $a <^{\fus_{n}} c <^{\fus_{n}} b$.
Consider $P\in\diamonds$ obtained by $\fus_{n}$ by blowing $c$ up to two points.
More precisely, let $c_0, c_1$ be two new elements, let $P =(\fus_{n} \setminus \set{c} )\cup \set{c_{0}, c_{1}}$, and define $\le^P, R^P$ by extending the corresponding relations on $\fus_n\setminus \set{c} $ requiring $a<^{P}c_{0}<^{P}c_{1}<^{P}b, a\binR^{P} c_{0}\binR^{P} c_{1} \binR^{P} b$.
Let $\fhi : P \to \fus_{n}$ be defined by:
\[
\fhi (d) = \begin{cases}
d & \text{if } d \in \fus_{n}, \\
c & \text{if } d\in \set{c_{0}, c_{1}}.
\end{cases}
\]
Then $ \fhi $ is an epimorphism by \Cref{iffchain}, and by \ref{itm:Ftwo} there exist $m>n$ and an epimorphism $\theta: \fus_{m} \to P$ such that $\fhi \theta = \fue^{m}_{n}$.
Let $a' \in (\fue^{m}_{n})^{-1}(a), b' \in (\fue^{m}_{n})^{-1}(b)$, then $\theta(a') = a, \theta(b')=b$.
If there was $c' \in \fus_{m}$ such that $a' \binR c' \binR b'$, then $\theta(c')$ should be $R^P$-connected to $a$ and $b$, but no such element exists in $P$.
By \Cref{lemmafour}, $(\fus_n, \fue_n^m)$ is therefore fine.

To prove irreducibility of the quotient map, by Lemma \ref{singletonsdenseiffirreducible} it suffices to show that for each $n \in \N$ and $a \in \fus_{n}$ there are $m>n$ and $b \in \fus_{m}$ such that $b' \binR  b$ implies $\fue^{m}_{n}(b') = a$.
To this end fix $n, a$ as above and define $P=\fus_{n} \sqcup \set{a_{0}, a_{1}, a_{2}}$ with $a_{0} \binR  a_{1} \binR  a_{2}$ and $a_{0} < a_1 < a_{2}$, so that $\set{a_{0}, a_{1}, a_{2}} \in \branches(P)$ and $P\in \diamonds$.
Let $\fhi: P\to \fus_{n}$ be the identity restricted to $\fus_{n}$ and $\fhi(a_{i})=a$ for $0\le i\le 2$.
By \Cref{iffchain}, $\fhi$ is an epimorphism and by \ref{itm:Ftwo} there exist $m>n$ and an epimorphism $\theta : \fus_{m} \to P$ such that $\fhi\theta = \fue^{m}_{n}$.
Let $b \in \theta^{-1}(a_{1})$ and $b' \binR  b$, then $\theta (b') \in \set{a_{0}, a_{1}, a_{2}}$, so $\fue^{m}_{n}(b)= a$.
\end{proof}

\subsection{A topological characterization of the \F fence}
The study of the quotient $ \quot{ \prespace }{R^{ \prespace }} $ is one of the main goals of this paper.
By \Cref{thmpidiamondfence}, $\quotofpre$ is a smooth fence.
We call \emph{\F fence} any space homeomorphic to $ \quotofpre $.

The following property of the \F fence is of crucial importance for its characterization.

\begin{lemma}\label{arcsproperty}
Let $ \fhi : \prespace \to P$ be an epimorphism onto some $P\in\diamonds$.
If $a, a' \in P$ with $a \le a'$, there is an arc component of $\quotofpre$ whose endpoints belong to $\interior(\cappello{a}_{ \fhi }), \interior( \cappello{a'}_{ \fhi })$, respectively.
\end{lemma}
\begin{proof}
Let $a_1, \ldots , a_{\ell }\in P$ be such that
\begin{align*}
  a<a_1< &\ldots <a_{\ell }<a', \\
 a\binR a_1\binR &\ldots \binR a_{\ell }\binR a'.
\end{align*}
Notice that $\ell=0$ if $a\binR a'$, in particular when $a=a'$.

Let $Q=P\sqcup\{ b, c, d_1, \ldots , d_{\ell }, b', c'\}\in\diamonds$, where
\begin{align*}
  b<c<d_1< &\ldots <d_{\ell }<b'<c', \\
  b\binR c\binR d_1\binR & \ldots \binR d_{\ell }\binR b'\binR c'.
\end{align*}
Let $\psi :Q\to P$ be the epimorphism defined as the identity on $P$ and by letting
\[ \left \{ \begin{array}{l}
\psi (b)=\psi (c)=a, \\
\psi (d_1)=a_1, \\
\ldots \\
\psi (d_{\ell })=a_{\ell }, \\
\psi (b')=\psi (c')=a'.
\end{array} \right . \]
By \ref{itm:Lthreeprime} there is an epimorphism $\theta : \prespace \to Q$ such that $\fhi =\psi\theta $.
Let $u, u'\in \prespace $ with $\theta (u)=b, \theta (u')=c', u\le u'$.
Given any $v\in \prespace $ with $v\le u$, if $w\binR v$, then $\theta(w)$ is either $b$ or $c$, so $ \fhi (w)=a$;
similarly, for any $v'\in \prespace$ with $u'\le v'$, if $w'\binR v'$, then $ \fhi (w')=a'$.
So, by \Cref{clmboundary}, $p(v)\in \interior ( \cappello{a}_{ \fhi }), p(v')\in \interior ( \cappello{a'}_{ \fhi })$.
This implies that the arc with endpoints $p(u), p(u')$ is contained in a connected component of $ \quotofpre $ with endpoints in $ \interior ( \cappello{a}_{ \fhi }), \interior ( \cappello{a'}_{ \fhi })$, respectively.
\end{proof}

The following theorem gives a topological characterization of the \F fence.

\begin{theorem}\label{weakcharacthm}
A smooth fence $Y$ is a \F fence if and only if for any two open sets $O, O' \subseteq Y$ which meet a common connected component  there is an arc component of $Y$ whose endpoints belong to $O, O'$, respectively.
\end{theorem}

The following lemmas are used in the proof of \Cref{weakcharacthm}.

\begin{lemma}\label{epiconditionslemma1}
Let $A, B, B'$ be HLOs and let $\fhi: B \to A$ and $\psi:B'\to A$ be $\mathcal L_{R}$-preserving maps such that $\psi[B'] \subseteq \fhi[B]$.
Let $a_0= \psi(\min B'), a_1= \psi(\max B')$ and $r = \max\setnew*{\Card{\fhi^{-1}(a)}}{a\in A}$.
If $\Card{\psi^{-1}(a)} \ge r$ for each $a\in \psi[B'] \setminus \set{a_0, a_1}$, then there exists an $\mathcal L_{R}$-preserving map $\theta:B'\to B$ such that $\fhi \theta = \psi$.
Moreover:
\begin{enumerate}
\item
if $\psi[B'] = \fhi[B]$ and $\Card{\psi^{-1}(a_0)}, \Card{\psi^{-1}(a_1)} \geq r$, then $\theta$ can be chosen to be surjective;
\item
	\label{itm:minimalB}
if $\psi[B'] = \fhi[B]$ and $\Card{\fhi^{-1}(a_0)} = \Card{\fhi^{-1}(a_1)} = 1$, then $\theta$ can be chosen to be surjective;
\item
	\label{itm:fixpoint}
if $a \in A$, $b \in \fhi^{-1}(a), b' \in \psi^{-1}(a)$ and
\[
\min \big\{\Card{\setnew*{c \in B'}{\psi(c)=a,c<b'} }, \Card{\setnew*{c \in B'}{\psi(c)=a,c>b'} } \big\} \ge r-1,
\]
then $\theta$ can be chosen such that $\theta(b') = b$.
\end{enumerate}
\end{lemma}

\begin{proof}
For each $a\in \psi[B'] \setminus \set{a_0, a_1}$ let $\theta$ map $\psi^{-1}(a)$ to $\fhi^{-1}(a)$ surjectively and monotonically.
If $\psi[B'] = \fhi[B]$ and $\Card{\psi^{-1}(a_0)}, \Card{\psi^{-1}(a_1)} \geq r$, doing the same for $\psi^{-1}(a_0), \psi^{-1}(a_1)$ provides a map onto $B$.
Otherwise, map all of $\psi^{-1}(a_0)$ to the maximal element of $\fhi^{-1}(a_0)$, and all of $\psi^{-1}(a_1)$ to the minimal element of $ \fhi^{-1}(a_1)$.
In the hypothesis of point \eqref{itm:minimalB}, this produces a surjective map on $B$.

\sloppypar
As for point \eqref{itm:fixpoint}, map $\setnew*{c \in B'}{\psi(c)=a, c\le b'}$, $\setnew*{c \in B'}{\psi(c)=a, c\ge b'}$ monotonically onto $\setnew*{c \in B}{\fhi(c)=a, c\le b}$, $\setnew*{c \in B}{\fhi(c)=a, c\ge b}$, respectively, so in particular $\theta(b')=b$.
\end{proof}

\begin{lemma}
	\label{enoughspace}
Let $(P_n, \fhi_n^m)$ be a fine projective sequence in $\diamonds$, with projective limit $ \mathbb P $, and the quotient map $p: \mathbb P \to \quot{ \mathbb P }{R^{ \mathbb P }} $ be irreducible.
Let $J^1, \dots, J^{\ell }$ be connected components of $\quot{ \mathbb P }{R^{ \mathbb P }}$.
For each $n \in \N$ and $1\le i\le \ell$, let $J^{i}_{n} = \fhi_{n}[p^{-1}(J^{i})]$ and $B^{i}_{n} \in \branches(P_{n})$ be such that $J^{i}_{n} \subseteq B^{i}_{n}$.
For any $n, r \in \N$, if the endpoints of the $J^{i}$'s belong to $\bigcup_{a \in P_{n}} \interior(\cappello{a}_{ \fhi_n})$, there is $m_{0}>n$ such that, for each $m \ge m_{0}$ and $1\le i\le \ell$:
\begin{enumerate}[label=(\alph*)]
\item
$\fhi^{m}_{n}[B^{i}_{m}] = J^{i}_{n}$,
\item
	\label{it:greaterthanr}
if $J^i$ is an arc, then $\Card{J^{i}_{m} \cap (\fhi^{m}_{n})^{-1}(a)} > r$ for each $a \in J^{i}_{n}$.
\end{enumerate}
\end{lemma}

\begin{proof}
We can suppose that the $J^{i}$'s are distinct.
Let $O_1, \dots, O_{\ell }$ be pairwise disjoint open neighborhoods of $J^1, \dots, J^{\ell }$, respectively, such that $O_{i} \subseteq \bigcup_{a \in J^{i}_{n}} \cappello{a}_{ \fhi_n}$, for $1\le i\le\ell$.
By \Cref{branchforcomponent}, there is $m'>n$ such that for $1\le i\le\ell$, one has $\bigcup_{a \in B^{i}_{m'}} \cappello{a}_{ \fhi_{m'}} \subseteq O_{i}$, that is, $\fhi^{m'}_{n}[B^{i}_{m'}] = J^{i}_{n}$.
It follows that for all $m>m'$ and $1\le i\le\ell$, one has $\fhi^m_{n}[B^{i}_{m}] = J^{i}_{n}$.
For $1\le i\le\ell$ such that $J^{i}$ is an arc, and each $a \in J^{i}_{n}$, the set $\cappello{a}_{ \fhi_n}\cap J^{i}$ has more than one element; since the mesh of $\cappello{P_{m}}_{ \fhi_m}$ goes to $0$, there exists $m_{0}>m'$ such that for all $m>m_{0}$ condition \ref{it:greaterthanr} is satisfied.
\end{proof}

\begin{proof}[Proof of \Cref{weakcharacthm}]
For the forward implication, it suffices to prove the conclusion for $ \quotofpre $.
Let $O, O' \subseteq \quotofpre$ be open sets which meet a common connected component $K$.
Let $n \in \N, a, a' \in \fus_{n}$ be such that
\[
\cappello{a}_{\fue_n}\subseteq O, \quad \cappello{a'}_{\fue_n}\subseteq O', \quad \interior ( \cappello{a}_{\fue_n})\cap K\ne\emptyset\ne \interior ( \cappello{a'}_{\fue_n})\cap K.
\]
It follows that $a, a'$ are $\le^{\fus_n}$-comparable, so by \Cref{arcsproperty} there is an arc component $J$ of $ \quotofpre $ whose endpoints belong to $ \interior ( \cappello{a}_{\fue_n}), \interior ( \cappello{a'}_{\fue_n})$, respectively, and so to $O, O'$, respectively.

Conversely, assume that for any open sets $O,O'\subseteq Y$ meeting a common connected component there is an arc component of $Y$ whose endpoints belong to $O,O'$, respectively.
Let $(P_n, \fhi_n^m)$ be the projective sequence defined as in the proof of \Cref{smoothfencespidiamond}, and let $ \mathbb Y $ be its projective limit.

It is then enough to prove that $ \mathbb Y $ is a projective Fra\"iss\'e limit of $\diamonds$.
To this end, by \Cref{provefundseq}, we must prove that given $P \in \diamonds $ and an epimorphism $\fhi: P \to P_{n}$, there are $m\ge n$ and an epimorphism $\psi:P_{m} \to P$ such that $\fhi \psi = \varphi^{m}_{n}$.
Let $r = \max \setnew*{\lvert \fhi^{-1}(C) \rvert}{C\in P_{n}}$ and $B^1, \dots, B^{\ell }$ be an enumeration of $\branches(P)$.

From $\min B^{i} \le^{P} \max B^{i}$ it follows that $\fhi(\min B^{i}) \le^{P_n} \fhi(\max B^{i})$, for $1\le i\le\ell$.
There is a connected component of $Y$ which meets the interior of both $\cappello{\fhi(\min B^{i})}_{\fhi_{n}}$, $\cappello{\fhi(\max B^{i})}_{\fhi_{n}}$, so by hypothesis there is an arc component $J^{i}$ of $Y$ whose endpoints belong to $\interior \cappello{\fhi(\min B^{i})}_{\fhi_{n}}$,
$\interior \cappello{\fhi(\max B^{i})}_{\fhi_{n}}$, respectively.
Notice that if $j \neq i$ is such that $\fhi[B^{j}] = \fhi[B^{i}]$, one can find a connected component $J^{j}$ disjoint from $J^{i}$, by applying the hypothesis to a couple of open sets $O \subseteq \cappello{\fhi(\min B^{i})}_{\fhi_{n}}, O' \subseteq \cappello{\fhi(\max B^{i})}_{\fhi_{n}}$ which intersect $J^{i}$ but avoid its endpoints.

By \Cref{enoughspace} there is $m_{0}>n$  such that for all $m \ge m_{0}$ there are $A^1, \dots, A^{\ell } \in \branches(P_{m})$ distinct such that, for $1\le i\le\ell$, one has $ \fhi^{m}_{n}[A^{i}] = \fhi[B^{i}]$ and $\Card{A^{i} \cap ( \fhi^{m}_{n})^{-1}(U)} \ge r$ for each $U \in \fhi[B^{i}]$.

On the other hand, since $\varphi $ is an epimorphism, for $m$ big enough it holds that for all $A \in \branches(P_{m})$ there is $B_{A} \in \branches(P)$ such that $ \fhi^{m}_{n}[A] \subseteq \fhi[B_{A}]$ and,
for every $U\in \fhi_n^m [A]$, one has $\card{( \fhi_n^m)^{-1}(U)\cap A} \ge r$.

So fix such an $m$, greater or equal to $m_{0}$.
We construct an epimorphism $\psi:P_{m} \to P$ such that $\fhi \psi = \fhi^{m}_{n}$, by defining its restriction on each $A \in \branches(P_{m})$.
For $1\le i\le\ell$, we use \Cref{epiconditionslemma1} to construct an $\mathcal L_{R}$-preserving function $\psi_{i}$ from $A^{i}$ onto $B^{i}$ such that $\fhi\psi_{i}= \fhi^{m}_{n} \restr{A^{i}}$.
Then, for each $A \in \branches(P_{m}) \setminus \setnew{A^{i}}{1\le i\le\ell}$, we again use \Cref{epiconditionslemma1} to find an $\mathcal L_{R}$-preserving function $\psi_{A}$ from $A$ to $B_{A}$ such that $\fhi\psi_{A} = \fhi^{m}_{n}\restr{A}$.
Then, defining $\psi =\bigcup_{i=1}^{\ell }\psi_i\cup\bigcup_{A\in \branches (P_m)\setminus \setnew{A^{i}}{1\le i\le\ell}}\psi_A $, it follows that $\fhi\psi = \fhi^{m}_{n}$ and, by \Cref{iffchain}, $\psi$ is an epimorphism.
\end{proof}

\subsection{Homogeneity properties of the \F fence}
\sloppypar
In this section we study some homogeneity properties of the \F fence, describing in particular its orbits under homeomorphisms.
We denote by $\homeo_{\le}(\quotofpre)$ the subgroup of $ \homeo(\quotofpre)$ of homeomorphisms which preserve $\le^{\prespace/R^{\prespace}}$.

\begin{theorem}
	\label{homogeneityofcomponents}
Let $J^1, \dots, J^{\ell }, I^1, \dots, I^{\ell }$ be two tuples of distinct connected components of $\quotofpre$.
Suppose that $J^1, \ldots , J^k, I^1, \ldots , I^k$ are arcs and $J^{k+1}, \ldots , J^{\ell }$, $I^{k+1}, \ldots , I^{\ell } $ are singletons, for some $k$ with $0\le k\le\ell $.
For $1\le i \le k$, let $x^{i} \in J^{i}, y^{i} \in I^{i}$ be points which are not endpoints.
Then there is $h \in \homeo_{\le}(\quotofpre)$ such that $h[J^{i}] = I^{i}$, for $1\le i\le\ell$, and $h(x^{i}) = y^{i}$ for $1\le i \le k$.
\end{theorem}

We obtain \Cref{homogeneityofcomponents} by proving in \Cref{backandforth} a strengthening of the converse of \Cref{provefundseq} for $( \fus_n, \fue_n^m)$ and using it in a back-and-forth argument which yields the desired homeomorphism.

\begin{lemma}
	\label{fixapoint}
Let $(P_n, \fhi_n^m)$ be a fine projective sequence in $\diamonds$, with projective limit $ \mathbb P $, and the quotient map $p: \mathbb P \to \quot{ \mathbb P }{R^{ \mathbb P }} $ be irreducible.
Let $x \in \quot{ \mathbb P }{R^{ \mathbb P }}$ be such that $p^{-1}(x)$ is a singleton which is neither $\le^{\mathbb P}$-minimal nor $\le^{\mathbb P}$-maximal.
For each $n \in \N$, let $\set{x_{n}} = \fhi_{n}[p^{-1}(x)]$.
For any $n, r \in \N$, there is $m_{0}>n$ such that for all $m>m_{0}$,
\[
\min \set{\Card{\setnew*{b \in P_{m}}{b<x_m, \fhi^{m}_{n}(b)=x_{n}} }, \Card{\setnew*{b \in P_{m}}{b>x_m, \fhi^{m}_{n}(b)=x_{n}} } } \ge r.
\]
\end{lemma}

\begin{proof}
Since $p^{-1}(x)$ is neither $\le^{\mathbb P}$-minimal nor $\le^{\mathbb P}$-maximal, there is $n_{0}>n$ such that $x_{n_{0}}$ is neither $\le^{P_{n_{0}}}$-minimal nor $\le^{P_{n_{0}}}$-maximal.
Let $a, a'$ be the $R^{P_{n_0}}$-neighbors of $x_{n_{0}}$ different from $x_{n_0}$.
By \Cref{clmboundary} it follows that $x \in \interior(\cappello{x_{n_{0}}}_{\fhi_{n_{0}}})$, so $x$ has positive distance from $\cappello{a}_{\fhi_{n_{0}}}$ and $\cappello{a'}_{\fhi_{n_{0}}}$.
By \Cref{generalprop}(2), there is $m_0>n_{0}$ for which the thesis holds.
\end{proof}

\begin{lemma}
	\label{backandforth}
Let $J^1, \dots, J^{\ell }$ be distinct connected components of $ \quotofpre $, such that $J^1,\ldots ,J^k$ are arcs and $J^{k+1},\ldots ,J^{\ell }$ are singletons, where $0\le k\le\ell $.
Assume that $p^{-1}(x)$ is a singleton, for any $x$ endpoint of some $J^i$.
For $1\le i \le k$, let $x^{i} \in J^{i}$ be a point which is not an endpoint, such that $p^{-1}(x^{i})$ is a singleton.
For each $n \in \N$, call $J^{i}_{n} = \fue_{n}[p^{-1}(J^{i})]$, and $\set{x^{i}_{n}} = \fue_{n}[p^{-1}(x^{i})]$.
Let $P \in  \diamonds$, and $\fhi: P \to \fus_{n}$ an epimorphism.
For $1\le i\le\ell $, let $I^{i} \subseteq P$ be $R$-connected and such that $\fhi[I^{i}] = J^{i}_{n}$; assume moreover that if $J^i$ is a singleton, then $I^{i}$ is a singleton as well.
For $1\le i \le k$, let $y^{i} \in \fhi^{-1}(x^{i}_{n})$.
Then there exist $m>n$ and an epimorphism $\psi: \fus_{m} \to P$ such that:
\begin{itemize}
\item $\psi[J^{i}_{m}] = I^{i}$ for $1\le i\le\ell$;
\item $\psi(x^{i}_{m}) = y^{i}$ for $1\le i \le k$; and
\item $\fhi \psi = \fue^{m}_{n}$.
\end{itemize}
\end{lemma}
\begin{proof}
Let $r = \max \setnew{ \Card{\fhi^{-1}(a)}}{a \in \fus_{n}}$.
For $1\le i\le\ell$ and $m \in \N$, let $B^{i}_{m} \in \branches(\fus_{m})$ be such that $J^{i}_{m} \subseteq B^{i}_{m}$.
Let $P' \in \diamonds$ be the structure obtained as the disjoint union of $\ell+1$ copies of $P$ and $\alpha: P' \to P$ be the epimorphism whose restriction to each copy of $P$ is the identity.
By \ref{itm:Ftwo} there are $m'>n$ and an epimorphism $\psi': \fus_{m'} \to P'$ such that $\fhi \alpha \psi' = \fue^{m'}_{n}$.
By \Cref{clmboundary} the endpoints of $J^{i}$ belong to $\bigcup_{a \in \fus_{m'}} \interior(\cappello{a}_{ \fue_{m'}})$, for $1\le i\le\ell$, so we can apply \Cref{enoughspace} to find $m_{0}>m'$ such that for all $m>m_{0}$ and $1\le i\le\ell$ we obtain that $\fue^{m}_{m'}[B^{i}_{m}] = J^{i}_{m'}$ and, if $J^i$ is an arc, $\Card{(\fue^{m}_{n})^{-1}(a) \cap J^{i}_{m}}>r$ for each $a \in J^{i}_{n}$.
For $1\le i \le k$, $p^{-1}(x^{i})$ is a singleton and is neither $\le^{\prespace}$-minimal nor $\le^{\prespace}$-maximal, so by \Cref{fixapoint} there is $m_{1} > m_{0}$ such that for all $m>m_{1}$ and $1\le i \le k$,
\begin{equation}
	\label{spaceabovebelow}
\min \set{\Card{\setnew*{b \in \fus_{m}}{b<x_m^i,\fue^{m}_{n}(b)=x^{i}_{n}}}, \Card{\setnew*{b \in \fus_{m}}{b>x_m^i,\fue^{m}_{n}(b)=x^{i}_{n}} }} \ge r.
\end{equation}
Fix such an $m>m_{1}$.
We use \Cref{epiconditionslemma1} to define, for $1\le i\le\ell$, an epimorphism $\psi_{i}: B^{i}_{m} \to I^{i}$ such that $\psi_{i}[J^{i}_{m}] = I^{i}$, $ \fhi \psi_i=\fue_n^m \restr{B_m^i}$, and such that, moreover, $\psi_{i}(x^{i}_{m}) = y^{i}$ when $1\le i \le k$.
Let $\psi: \fus_{m} \to P$ be defined by
\[
\psi (b)= \left \{ \begin{array}{lcl}
\alpha\psi' \fue_{m'}^m(b) & \text{if} & b\notin\bigcup_{i=1}^{\ell }B_m^i \\
\psi_i(b) & \text{if} & b\in B_m^i
\end{array} \right .
.
\]
Then $\fhi \psi = \fue^{m}_{n}$ and $\psi$ is an epimorphism.
Indeed, $\psi$ is $\mL_{R}$-preserving by construction and for each $B \in \branches(P)$ there is $C \in \branches(\fus_{m})$ such that $\psi'\fue_{m'}^m[C]$ equals one of the copies of $B$ in $P'$, as there are more copies of $B$ in $P'$ than maximal chains of $\fus_{m}$ on which $\psi$ differs from $\alpha\psi'\fue_{m'}^m$.
\end{proof}

The connected components of \Cref{homogeneityofcomponents} might not satisfy the hypotheses of \Cref{backandforth}, since some of the endpoints may be non-singleton $R^{ \prespace }$-classes, so we cannot apply \Cref{backandforth} directly.
Therefore we first need the following lemma.

\begin{lemma}\label{identifyendpoints}
Let $\sim \subseteq R^{\prespace}$ be an equivalence relation on $\prespace$ which is the equality but on finitely many points.
Then $\altprespace$ with the induced $ \mL_R$-structure is isomorphic to $\prespace$.
\end{lemma}
\begin{proof}
Let $\ell$ be the number of $\sim$-equivalence classes of cardinality greater than $1$, that is, by \Cref{fineclassesconvex}, of cardinality $2$.
Denote these equivalence classes by $\set{x_1, x'_1}, \dots, \set{x_{\ell }, x'_{\ell }}$.
To prove that  $\altprespace$ is isomorphic to $\prespace$ we show that $\altprespace$ satisfies properties \ref{itm:Lone}, \ref{itm:Ltwo} and \ref{itm:Lthreeprime}.
Inductively, it is enough to prove the assertion for $\ell =1$.
Notice also that the quotient map $q: \prespace \to \altprespace$ is an epimorphism.

Property \ref{itm:Lone} follows from \ref{itm:Lthreeprime} by considering, for any $P\in \diamonds$, epimorphisms from $\altprespace$ and $P$ to a structure in $\diamonds$ with one point.

To check that \ref{itm:Lthreeprime} holds, fix $P, Q \in \diamonds$ and epimorphisms $\psi: \altprespace \to P, \fhi: Q \to P$ with the objective of finding an epimorphism $\theta: \altprespace \to Q$ such that $\fhi \theta = \psi$.
Let $Q' \in \diamonds$ be the structure obtained from $Q$ by substituting each $a \in Q$ with a chain $ \set{a_0, a_1} $ of length $2$.
In other words:
\begin{itemize}
\item $Q'= \setnew{a_0, a_1}{a\in Q} $;
\item $R^{Q'}$ is the smallest reflexive and symmetric relation such that
\begin{itemize}
\item $a_{0} \binR^{Q'} a_1$ for every $a\in Q$,
\item $a_1 \binR^{Q'} a'_{0}$ whenever $a \binR^{Q} a'$, with $a <^{Q} a'$;
\end{itemize}
\item $a_i\le^{Q'}a'_j$ if and only if either $a=a', i\le j$, or $a<^Qa'$.
\end{itemize}
Let $\chi : Q' \to Q$ be the epimorphism $a_{i} \mapsto a$.
By \ref{itm:Lthreeprime} for $\prespace$ there exists $\theta': \prespace \to Q'$ such that $\fhi \chi \theta' = \psi q$.
Let $C=\theta'[\set{ x_1, x'_1}]$.
Let $\chi': Q' \to Q$ be defined as
\[
\chi'(a_{i})= \begin{cases}
	a & \text{ if } a_{i} \not \in C, \\
	\chi(\max C) & \text{ if } a_{i}  \in C.
\end{cases}
\]
Then $\chi'$ is an epimorphism using \Cref{iffchain}, which is applicable as $\forall a\in Q'\ \chi'(a_0)=a$.
Define $\theta (y)= \chi' \theta'(x)$ for any $x\in q^{-1}(y)$.
This is well defined as $\chi'\theta'(x_1)= \chi'\theta'(x'_1)$, and is the required epimorphism: continuity holds since for each $a \in Q$, the set $(\chi' \theta')^{-1}(a)$ is a clopen $\sim$-invariant subset of $\prespace$, so $q[(\chi' \theta')^{-1}(a)]=\theta^{-1}(a)$ is clopen in $\altprespace$.

For \ref{itm:Ltwo} let $ \set{V_1, \dots, V_r} $ be a clopen partition of $\altprespace$.
Consider the induced clopen partition $ \set{q^{-1}(V_1), \dots, q^{-1}(V_r)} $ of $\prespace$.
By \ref{itm:Ltwo} for $\prespace$, there exist $P' \in \diamonds$ and an epimorphism $\fhi': \prespace \to P'$ which refines the partition.
Let $P\in \diamonds$ be the quotient of $P'$ which identifies $a, a'$ if and only if $a = a'$ or $a, a' \in \fhi'[\set{x_1, x'_1}]$.
Then the quotient map $\psi:P' \to P$ is an epimorphism, so $\fhi (y)= \psi \fhi'(x)$ for any $x\in q^{-1}(y)$ is a well defined epimorphism.
Since $\psi \fhi'$ refines $ \set{q^{-1}(V_1),\ldots ,q^{-1}(V_r)} $, it follows that $\fhi$ refines $ \set{V_1, \dots, V_r} $.
\end{proof}

\begin{proof}[Proof of \Cref{homogeneityofcomponents}]
By \Cref{identifyendpoints}, up to considering an isomorphic structure, we can assume that the preimages of the endpoints of all the $J^{i}$'s and $I^{i}$'s under the quotient map $p: \prespace \to \quotofpre$ are singletons, as well as the preimages of the $x^{i}$'s and $y^{i}$'s.

For $1 \le i \le \ell$, let $J^{i}_{\infty} = p^{-1}(J^{i}), I^{i}_{\infty} = p^{-1}(I^{i})$;
for $1\le i \le k$, let $\set{x^{i}_{\infty}} = p^{-1}(x^{i}), \set{y^{i}_{\infty}} = p^{-1}(y^{i})$.
For each $n \in \N$, for $1 \le i \le \ell$, let $J^{i}_{n} = \fue_{n}[J^{i}_{\infty}], I^{i}_{n} = \fue_{n}[I^{i}_{\infty}]$; for $1\le i \le k$, let $x^{i}_{n}= \fue_{n}(x^{i}_{\infty}), y^{i}_{n}= \fue_{n}(y^{i}_{\infty})$.
When $J^i$ (equivalently, $I^i$) is a singleton, then $J_n^i, I_n^i$ are singletons for every $n\in \N $.

Let $n_{0} = m_{0} = 0$ and $\fhi_{0}: \fus_{m_{0}} \to \fus_{n_{0}}$ be the identity.
As $\fus_{0}$ consists of a single point, all the hypotheses of \Cref{backandforth} are satisfied where $n, P, I^i, y^{i}, \fhi $ of the lemma are $0, F_0, I_0^i, y^{i}_{0}, \fhi_0$, respectively.
Suppose that $n_{j}, m_{j}, \fhi_{j}: \fus_{m_{j}} \to \fus_{n_{j}}$ have been defined and are such that $\fhi_{j}[I^{i}_{m_{j}}] = J^{i}_{n_{j}}$ for $1 \le i  \le \ell$, and $\fhi_{j}(y^{i}_{m_{j}})=x^{i}_{n_{j}}$ for $1\le i \le k$.
By \Cref{backandforth} there exist $n_{j+1}> n_{j}$ and $\psi_{j}: \fus_{n_{j+1}} \to \fus_{m_{j}}$ such that $\fhi_{j}\psi_{j} = \fue^{n_{j+1}}_{n_{j}}$, $\psi_{j}[J^{i}_{n_{j+1}}] = I^{i}_{m_{j}}$, for $1 \le i \le \ell$, and $\psi_{j}(x^{i}_{n_{j+1}}) = y^{i}_{m_{j}}$, for $1\le i \le k$.
Now $\fus_{m_{j}}, \fus_{n_{j+1}}$ and $\psi_{j}$ satisfy the hypotheses of \Cref{backandforth} with the roles of the $I$'s and $J$'s reversed, so there exist $m_{j+1}>m_{j}$ and $\fhi_{j+1}: \fus_{m_{j+1}} \to \fus_{n_{j+1}}$ such that $\psi_{j}\fhi_{j+1} = \fue^{m_{j+1}}_{m_{j}}$, $\fhi_{j+1}[I^{i}_{m_{j+1}}] = J^{i}_{n_{j+1}}$ for $1 \le i \le \ell$, and $\fhi_{j+1}(y^{i}_{m_{j+1}})=x^{i}_{n_{j+1}}$, for $1\le i \le k$.

Let $\fhi, \psi: \prespace \to \prespace$ be the unique epimorphisms such that for each $j \in \N$, $\fue_{n_{j}} \fhi = \fhi_{j} \fue_{m_{j}}$ and $\fue_{m_{j}} \psi = \psi_{j} \fue_{n_{j+1}}$.
Then $\fhi \psi$ and $\psi \fhi$ are the identity, so $\fhi, \psi \in \aut(\prespace)$.
As for each  $j \in \N$, $\fue_{m_{j}} \psi[J^{i}_{\infty}] = \psi_{j} \fue_{n_{j+1}}[J^{i}_{\infty}] = \psi_{j}[J^{i}_{n_{j+1}}] = I^{i}_{m_{j}}$ for $1 \le i \le \ell$, it follows that $\psi[J^{i}_{\infty}] = I^{i}_{\infty}$;
from $\fue_{m_{j}} \psi(x^{i}_{\infty}) = \psi_{j} \fue_{n_{j+1}}(x^{i}_{\infty}) = \psi_{j}(x^{i}_{n_{j+1}}) = y^{i}_{m_{j}}$, it follows that $\psi(x^{i}_{\infty}) = y^{i}_{\infty}$, for $1\le i \le k$.
Let $h: \quotofpre \to \quotofpre $ be defined by $h(x)=p\psi (u)$ for any $u\in p^{-1}(x)$.
Then $h\in \homeo_{\le}(\quotofpre)$ and $h[J^{i}] = I^{i}$, for $1 \le i \le \ell$, and $h(x^{i}) = y^{i}$ for $1\le i \le k$.
\end{proof}

To lighten notation, let $ \mathfrak L = \mathfrak L_{\le^{ \mathbb F /R^{ \mathbb F }}} \left (
\quotofpre
\right ),
\mathfrak U = \mathfrak U_{\le^{ \mathbb F /R^{ \mathbb F }}} \left (
\quotofpre
\right )
$.

\begin{lemma}
	\label{landuhomeo}
There is $h \in \homeo(\quotofpre)$ which switches $\mathfrak U$ and $\mathfrak L$.
\end{lemma}
\begin{proof}
For any  $\mL_{R}$-structure $A$, let $A^{*}$ be the $\mL_{R}$-structure with the same support as $A$, with $R^{A^{*}} = R^{A}$ and $u \le^{A^{*}} u'$ if and only if $u' \le^{A} u$.
Then $(A^{*})^{*} = A$ and a function $\fhi: B \to A$ is an epimorphism from $B$ to $A$ if and only if it is an epimorphism from $B^{*}$ to $A^{*}$.
Now, if $A \in \diamonds$, then $A^{*} \in \diamonds$, so it is straightforward to check that
\ref{itm:Lone}, \ref{itm:Ltwo}, \ref{itm:Lthree} hold for $\prespace^{*}$.
It follows that $\prespace^{*}$ is the projective Fraïssé limit of $\diamonds$ and thus that it is isomorphic to $\prespace$, via an isomorphism $\alpha: \prespace \to \prespace^{*}$.
Let $h: \quotofpre \to \quotofpre $ be defined by letting $h(x)=p\alpha (u)$ for any $u\in p^{-1}(x)$.
Then $h$ is the required homeomorphism.
\end{proof}

\begin{corollary}
	\label{onethirdhomo}
The \F fence is $\nicefrac{1}{3}$-homogeneous.
The orbits of the action of $\homeo(\quotofpre)$ on $\quotofpre$ are $\mathfrak L \cap \mathfrak U$, $ \mathfrak L \symdif \mathfrak U$, and $\quotofpre \setminus (\mathfrak L \cup \mathfrak U)$.
\end{corollary}
\begin{proof}
The above subspaces are clearly invariant under homeomorphisms.
We conclude by \Cref{homogeneityofcomponents} and \Cref{landuhomeo}.
\end{proof}

The \F fence also enjoys a different kind of homogeneity property, namely that of $h$-homogeneity.

\begin{proposition}
	\label{newspacehhomo}
The \F fence is $h$-homogeneous.
\end{proposition}

\begin{proof}
Fix a nonempty clopen subset $U$ of $\quotofpre$.
By \Cref{clopensareunionofbranches}, there is $n_{0} \in \N$ such that for all $n \ge n_{0}$, there is $S_{n} \subseteq \branches(P_{n})$ for which $U = \bigcup_{a\in\bigcup S_n}\cappello{a}_{\fue_n}$.
Let $Q_{n} = \bigcup S_{n}$.
We prove that $(Q_{n}, \fue^{m}_{n} \restr{Q_m})_{n \ge n_{0}}$ is a fundamental sequence in $\diamonds$, thus showing that $p^{-1}(U)$, with the $ \mL_R$-structure inherited from $ \prespace $, is isomorphic to $\prespace$, which yields the result.

Let $n\ge n_{0}$, $P \in \diamonds$ and $\fhi: P \to Q_{n}$.
Let $P' = P \sqcup (\fus_{n} \setminus Q_{n})$ and $\fhi' : P' \to \fus_{n}$ be $\fhi$ on $P$ and the identity on $\fus_{n} \setminus Q_{n}$.
Since $Q_{n}$ is $R^{P_n}$-invariant in $\fus_{n}$ and $\fhi$ is an epimorphism, so is $\fhi'$, by \Cref{iffchain}.
By \ref{itm:Ftwo} there are $m\ge n$ and an epimorphism $\psi': \fus_{m} \to P'$ such that $\fhi' \psi' = \fue^{m}_{n}$.
We see that $(\fue^{m}_{n})^{-1}(Q_{n}) = Q_{m}$.
Indeed, $\fue_{m}^{-1}(Q_{m}) = \fue_{n}^{-1}(Q_{n}) = p^{-1}(U)$, so $Q_{m} \subseteq (\fue^{m}_{n})^{-1}(Q_{n}) \subseteq \fue_{m}[\fue_{n}^{-1}(Q_{n})] = \fue_{m}[p^{-1}(U)] = Q_{m}$.
Therefore $(\psi')^{-1}(P)=  Q_{m}$, so $\psi = \psi' \restr{Q_{m}} : Q_{m} \to P$ is an epimorphism such that $\fhi \psi = \fue^{m}_{n} \restr{Q_{m}}$.
We conclude by \Cref{provefundseq}.
\end{proof}

\subsection{A strong universality property of the \F fence}

\Cref{smoothembedd} shows that any smooth fence embeds in the Cantor fence.
We show a stronger universality property for the \F fence, namely that any smooth fence embeds in the \F fence via a map which preserves endpoints.

\begin{theorem}
	\label{thm:FraisseFenceUniversal}
For any smooth fence $Y$ there is an embedding $f: Y \to \quotofpre$ such that $f[\mathrm{E}(Y)] \subseteq \mathrm{E}(\quotofpre)$.
Moreover, fixing a strongly compatible order $\preceq $ on $Y$, the embedding $f$ can be constructed so that $f[ \mathfrak L (Y)]\subseteq \mathfrak L ,f[ \mathfrak U (Y)]\subseteq \mathfrak U $.
\end{theorem}
\begin{proof}
By \Cref{smoothfencespidiamond} there is a projective sequence $(P_{n}, \fhi_{n}^m)$, with projective limit $\mathbb P$ such that $\quot{\mathbb P}{R^{\mathbb P}}$ is homeomorphic to $Y$, via $h: \quot{\mathbb P}{R^{\mathbb P}} \to Y$; moreover, $h$ is an isomorphism between $\le^{ \mathbb P /R^{ \mathbb P }}$ and $\preceq $.
Therefore it is enough to prove the assertion for $( \quot{ \mathbb P }{R^{ \mathbb P }} ,\le^{ \mathbb P /R^{ \mathbb P }})$.

Let $q: \mathbb P \to \quot{\mathbb P}{R^{\mathbb P}}$ be the quotient map.
We procede by induction to define a topological $ \mL_R$-structure $ \mathbb P'\subseteq \prespace $ isomorphic to $ \mathbb P $.
Let $a_{0} \in \fus_{0}$, $P'_{0} = \set{a_{0}} \subseteq \fus_{0}$, and $\theta_{0}: P_{0} \to P'_{0}$ be the unique epimorphism.

Suppose one has defined $i_{n}, j_{n} \in \N$, $P'_{n} \subseteq \fus_{i_{n}}$; assume also that, with the induced structure, $P'_n\in \diamonds $ and there is an epimorphism $\theta_{n}: P_{j_{n}} \to P'_{n}$.
Let $F'_{n} = \fus_{i_{n}} \sqcup P_{j_{n}}$ and $\theta'_{n}: F'_{n} \to \fus_{i_{n}}$ be the identity on $\fus_{i_{n}}$ and $\theta_{n}$ on $P_{j_{n}}$.
By \ref{itm:Ftwo} there are $i_{n+1} > i_{n}$ and an epimorphism $\psi_{n}: \fus_{i_{n+1}} \to F'_{n}$ such that $\fue^{i_{n+1}}_{i_{n}} = \theta'_{n} \psi_{n}$.
Then $\psi_n^{-1}(P_{j_n})$ is an $R^{F_{i_{n+1}}}$-invariant subset of $F_{i_{n+1}}$, that is, the union of a subset of $ \branches (F_{i_{n+1}})$.
Let $P'_{n+1} \subseteq \psi_n^{-1}(P_{j_n})$ be in $ \diamonds $, with respect to the induced $ \mL_R$-structure, and minimal, under inclusion, with the property that $\psi_{n} \restr{P'_{n+1}}$ is an epimorphism onto $P_{j_{n}}$.
This means that there is a bijection $g: \branches(P_{j_{n}}) \to \branches(P'_{n+1})$ such that  $\psi_{n}[g(A)] = A$ and $ \Card{\psi^{-1}(\min A)\cap g(A)} = \Card{\psi^{-1}(\max A)\cap g(A)} =1$, for any $A \in  \branches(P_{j_{n}})$.
Let $r = \max \setnew{ \card{ \psi_{n}^{-1}(a) \cap g(A)}}{a \in A, A \in \branches(P_{j_{n}})}$.

Since the sequence $(P_{n}, \fhi_{n}^m)$ is fine, by \Cref{lemmafour}, there is $j_{n+1}> j_{n}$ such that for all $a, b \in P_{j_{n}}$ with $d_{R^{P_{j_{n}}}}(a, b) = 2$, and all $a' \in (\fhi^{j_{n+1}}_{j_{n}})^{-1}(a), b' \in (\fhi^{j_{n+1}}_{j_{n}})^{-1}(b)$, it holds that $d_{R^{P_{j_{n+1}}}}(a', b') \ge r+1$; this means that if $B$ is an $R^{P_{j_{n+1}}}$-connected chain in $P_{j_{n+1}}$ and $c\in \fhi_{j_n}^{j_{n+1}}[B]\setminus\{\min \fhi_{j_n}^{j_{n+1}}[B],\max \fhi_{j_n}^{j_{n+1}}[B]\} $, then $ \Card{( \fhi_{j_n}^{j_{n+1}})^{-1}(c)\cap B} \ge r$.
We find an epimorphism $\theta_{n+1}: P_{j_{n+1}} \to P'_{n+1}$ by defining it on each maximal chain.
Fix $B \in \branches(P_{j_{n+1}})$.
Let $A \in \branches(P_{j_{n}})$ be such that $\fhi^{j_{n+1}}_{j_{n}}[B] \subseteq A$ and $B' \subseteq g(A)$ be the minimal subset such that  $\psi_{n}[B'] = \fhi^{j_{n+1}}_{j_{n}}[B]$.
Then $B, \fhi^{j_{n+1}}_{j_{n}}[B]$ and $B'$ satisfy the hypothesis of \Cref{epiconditionslemma1}\eqref{itm:minimalB}, so there is an epimorphism $\theta_{B}: B \to B'$ such that $\psi_{n}\theta_{B} = \fhi^{j_{n+1}}_{j_{n}} \restr{B}$.
Let $\theta_{n+1}=\bigcup_{B \in \branches(P_{j_{n+1}})}\theta_{B}$.
Then $\theta_{n+1}$ is an epimorphism by \Cref{iffchain}: for each $A \in \branches(P_{j_{n}})$, there is $B \in  \branches(P_{j_{n+1}})$ with $\fhi^{j_{n+1}}_{j_{n}}[B]=A$, so $\theta_{n+1}[B]\subseteq g(A)$, and by minimality of $g(A)$ it follows that $\theta_{n+1}[B] = g(A)$.
Note that $\psi_n \restr{P'_{n+1}} \theta_{n+1}= \fhi_{j_n}^{j_{n+1}}$.

The functions $ \fue_{i_n}^{i_{n+1}} \restr{P'_{n+1}} :P'_{n+1}\to P'_n$ are epimorphisms, so $\mathbb P' = \setnew{ u \in \prespace}{\forall n \in \N \ \fue_{i_{n}}(u) \in P'_{n}}$, with the induced $\mL_{R}$-structure is the limit of the projective sequence $(P'_n, \fue_{i_n}^{i_m} \restr{P'_m} )$.
Since $ \fue_{i_n}^{i_{n+1}} \restr{P'_{n+1}} \theta_{n+1}=\theta_n\psi_n \restr{P'_{n+1}} \theta_{n+1}=\theta_n \fhi_{j_n}^{j_{n+1}}$, let $\theta: \mathbb P \to \mathbb P'$ be the unique epimorphism such that for each $n \in \N$, $\fue_{i_{n}} \restr{ \mathbb P'} \theta = \theta_{n+1} \fhi_{j_{n+1}}$.
Similarly, as $ \fhi_{j_n}^{j_{n+1}}\psi_{n+1} \restr{P'_{n+2}} =\psi_n \restr{P'_{n+1}} \theta_{n+1}\psi_{n+1} \restr{P'_{n+2}} =\psi_n \restr{P'_{n+1}} \gamma_{i_{n+1}}^{i_{n+2}} \restr{P'_{n+2}} $, let $\psi: \mathbb P' \to \mathbb P$ be the unique epimorphism such that for each $n \in \N$, $\fhi_{j_{n}} \psi = \psi_{n} \restr{P'_{n+1}} \fue_{i_{n+1}} \restr{\mathbb P'}$.
Then $\theta \psi$ and $\psi \theta$ are the identity, so $\theta, \psi$ are isomorphisms.
Let $f: \quot{ \mathbb P}{R^{ \mathbb P }} \to \quotofpre$ be defined by letting $f(x)= p \theta (w)$ for any $w\in q^{-1}(x)$.
Then $f$ is an embedding.

We show that $\le^{\prespace}$-maximal (respectively, $\le^{\prespace}$-minimal) points of $\mathbb P'$ are $\le^{\prespace}$-maximal (respectively, $\le^{\prespace}$-minimal) in $\prespace$, thus concluding the proof.
To this end, let $u \in \mathbb P'$ be $\le^{\prespace}$-maximal in $\mathbb P'$ and fix $n \in \N$.
Let $a_m= \max \setnew{a \in P'_{m}}{ \fue_{i_{m}}(u) \le a}$;
by \Cref{endpointprop}, there is $m>n$ such that $\fue^{i_{m}}_{i_{n}}(a_{m})= \fue_{i_{n}}(u)$.
By minimality of $P'_{m}$, it follows that $\psi_{m-1}(a_m)$ is $\le^{F'_{m-1}}$-maximal, so for any $a \in \fus_{i_{m}}$ with $a_m\le a$, we have $\psi_{m-1}(a) = \psi_{m-1}(a_m)$, so $\fue^{i_{m}}_{i_{m-1}}(a) = \fue^{i_{m}}_{i_{m-1}}(a_m)$.
It holds therefore that $\fue^{i_{m}}_{i_{n}}(a) = \fue^{i_{m}}_{i_{n}}(a_m)= \fue_{i_{n}}(u)$.
By \Cref{endpointprop}, it follows that $u$ is $\le^{\prespace}$-maximal in $\prespace$.
The case for $\le^{\prespace}$-minimal points is analogous.
\end{proof}

Property \ref{itm:Lone} for $\prespace$ gives us another universality result for $\quotofpre$, namely \emph{projective universality}.

\begin{proposition}
For any smooth fence $Y$ with a strongly compatible order $\preceq$, there is a continuous surjection $f: \quotofpre \to Y$ such that $f \times f [\le^{\prespace/R^{\prespace}}] = \preceq$.
\end{proposition}
\begin{proof}
By \Cref{smoothfencespidiamond} we can assume that $Y= \quot{ \mathbb P }{R^{ \mathbb P }}$ for some fine projective sequence $(P_n, \fhi_n^m)$ in $\diamonds$ with projective limit $ \mathbb P $, and that $\preceq$ is $\le^{ \mathbb P /R^{ \mathbb P }}$.
Denote by $p: \prespace \to \quotofpre$ and $q: \mathbb P \to \quot{\mathbb P}{R^{\mathbb P}}$, the respective quotients maps.
By \cite{Irwin2006}*{Proposition 2.6}, there is an epimorphism $\varphi: \prespace \to \mathbb P$.
By \cite{Irwin2006}*{Lemma 4.5(i)} there is a continuous surjection $f: \quotofpre \to \quot{\mathbb P}{R^{\mathbb P}}$ such that $fp = q\varphi$.
It follows from the fact that $\varphi$ is an epimorphism that $f \times f [\le^{\prespace/R^{\prespace}}] = \preceq$.
\end{proof}

\begin{remark}
Property \ref{itm:Lthree} together with a strengthening of property \ref{itm:Ltwo}, give us \emph{approximate projective homogeneity} of the Fraïssé fence with respect to smooth fences.
Namely, for every smooth fence $Y$ with a strongly compatible order $\preceq$, any two continuous surjections $f_{0}, f_{1}: \quotofpre \to Y$ such that $f_{i} \times f [\le^{\prespace/R^{\prespace}}] = \preceq$, and any open cover $\mathcal V$ of $Y$, there is $h \in \homeo_{\le}(\quotofpre)$ such that $f_{0} h$ and $f_{1}$ are $\mathcal V$-close --- that is, for each $x \in \quotofpre$ there is $V \in \mathcal V$ such that $f_{0} h
(x), f_{1}(x) \in V$.
This was proved by the first author in his thesis, see \cite{BassoThesis}*{Corollary 4.5.2}.
\end{remark}

\subsection{Spaces of endpoints of the \F fence}

By \Cref{landuhomeo}, $ \mathfrak L $ and $ \mathfrak U$ are homeomorphic.
It also follows from that lemma that $\mathfrak U \setminus \mathfrak L , \mathfrak L \setminus \mathfrak U$ are homeomorphic.
We therefore state the results in this section solely in terms of $\mathfrak U, \mathfrak L \cap \mathfrak U$, and $\mathfrak U \setminus \mathfrak L$, the latter of which we denote by $\mathfrak M$.
In \Cref{riassunto} below we see that $\mathfrak L \cap \mathfrak U$ is homeomorphic to the Baire space $\N^{\N}$.

\begin{corollary}
	\label{homogeneous}
$\mathfrak M$ and $\mathfrak L \cap \mathfrak U$ are $n$-homogeneous for every $n\ge 1$.
\end{corollary}

\begin{proof}
From \Cref{homogeneityofcomponents}.
\end{proof}

\begin{proposition}
	\label{notzerodim}
$\mathfrak M$ is one-dimensional.
\end{proposition}
\begin{proof}
As $\mathfrak M$ is a subset of a one-dimensional space, its dimension is at most one.
We now show that it is at least one.
Let $x \in \mathfrak M $ and $J$ be the arc component of $\quotofpre$ to which it belongs.
Let $O$ be an open neighborhood of $x$ in $\quotofpre$ such that $J \not \subseteq \closure(O)$.
Let $n_{0}$ be such that there are $B_{0} \in \branches(\fus_{n_{0}})$ and $a_{0} \in B_{0}$ with
\[
J \subseteq \bigcup_{a\in B_0} \cappello{a}_{\fue_{n_0}}, \quad \cappello{ \max B_{0}}_{\fue_{n_0}}\subseteq O, \quad \cappello{a_0}_{\fue_{n_0}}\subseteq \quotofpre \setminus \closure(O),
\]
which exists by \Cref{upperlower}.
Let $a'_{0} \in B_{0}$ be the minimum such that $\bigcup_{ a \ge a'_0} \cappello{a}_{\fue_{n_0}}\subseteq O$.
Notice that $a_{0} < a'_{0}$.

Suppose one has defined $n_i\in \N, B_i\in \branches (\fus_{n_i}), a_i, a'_i\in B_i$, with $a_i< a'_i$.
By \Cref{arcsproperty} there exists an arc component $J_{i}$ of $\quotofpre$ whose endpoints belong to $\interior( \cappello{a_i}_{\fue_{n_i}}), \interior( \cappello{a'_i}_{\fue_{n_i}})$, respectively.
By \Cref{upperlower} there are $n_{i+1}>n_{i}$ and $B_{i+1} \in \branches(\fus_{n_{i+1}})$ such that
\begin{align*}
  J_i\subseteq\bigcup_{a\in B_{i+1}} \cappello{a}_{\fue_{n_{i+1}}}\ &\subseteq \bigcup_{a\in B_i} \cappello{a}_{\fue_{n_{i}}}, \\
 \cappello{\max B_{i+1}}_{\fue_{n_{i+1}}} &\subseteq \cappello{a'_i}_{\fue_{n_i}}.
\end{align*}
Choose $a_{i+1} \in B_{i+1}$ such that $ \cappello{a_{i+1}}_{\fue_{n_{i+1}}}\subseteq \cappello{a_i}_{\fue_{n_i}}$ and
let $a'_{i+1} \in B_{i+1}$ be the minimum such that $ \bigcup_{ a \ge a'_{i+1}} \cappello{a}_{\fue_{n_{i+1}}}\subseteq O$, so in particular $a_{i+1} < a'_{i+1}$.
Since the mesh of $( \cappello{\fus_n}_{\fue_n})_{n \in \N}$ goes to $0$, we can furthermore choose $n_{i+1}$ so that $ \cappello{a'_{i+1}}_{\fue_{n_{i+1}}}\nsubseteq \cappello{a'_i}_{\fue_{n_i}}$, so that in particular $a'_{i+1}\ne\max B_{i+1}$.

Let $K = \bigcap_{i \in \N} \bigcup_{a\in B_i} \cappello{a}_{\fue_{n_i}}=\lim_{i\rightarrow\infty }\bigcup_{a\in B_i} \cappello{a}_{\fue_{n_i}}$.
By \Cref{quotlimitconnected}, $K$ is connected, call $y$ its maximum.
We prove that
\[
y \in  \mathfrak M \quad \text{and} \quad y \in \closure_{ \mathfrak M } \left (
O\cap \mathfrak M  \right ) \setminus O,
\]
which concludes the proof.

Since $\bigcup_{a\in B_i} \cappello{a}_{\fue_{n_i}}\cap \cappello{a_{0}}_{\fue_{n_0}}\ne\emptyset $ for each $i$, it follows that $K \cap \cappello{a_{0}}_{\fue_{n_0}}\neq \emptyset$, so $y \not \in \mathfrak L $.
Suppose there exists $y' \in \quotofpre$, $y<^{ \mathbb F /R^{ \mathbb F }}y'$.
Let $U$ be an open set containing $K$ while avoiding $y'$.
There thus is $i \in \N$ such that $\bigcup_{a\in B_i} \cappello{a}_{\fue_{n_i}}\subseteq U$.
For each $a' \in \fus_{n_{i}}$ with $y'\in \cappello{a'}_{\fue_{n_i}}$, it follows that $a' \not \in B_{i}$ as $ \cappello{a'}_{\fue_{n_i}}\not \subseteq U$.
But $y \le^{\mathbb P/R^{\mathbb P}} y'$ implies $a\le a'$ for some $a \in B_{i}$, a contradiction.
So $y \in \mathfrak M $.

Since $ \cappello {a'_i}_{\fue_{n_i}}\subseteq O$ and $\max J_i\in \interior \left (
\cappello{a'_i}_{\fue_{n_i}}
\right ) $
for each $i \in \N$, it follows that $y \in \closure_{ \mathfrak M } \left (
O\cap  \mathfrak M
\right ) $.
Suppose that $y \in O$.
Since $y$ has positive distance from $K \setminus O$, there exists $i \in \N$ such that $y \not \in\bigcup \setnew{\cappello{a}_{\fue_{n_i}}}{ a \in B_{i}, a \le a'_{i}}$, as $a'_{i}$ is the minimum element of $B_{i}$ such that $\bigcup_{ a \ge a'_{i}} \cappello{a}_{\fue_{n_i}}\subseteq O$, and the diameter of the $ \cappello{a'_i}_{\fue_{n_i}}$ goes to $0$.
It follows that $y \not \in \bigcup_{a\in B_{i+1}} \cappello{a}_{\fue_{n_{i+1}}}$ as $\bigcup_{a\in B_{i+1}} \cappello{a}_{\fue_{n_{i+1}}}\subseteq \bigcup \setnew{\cappello{a}_{\fue_{n_i}}}{ a \in B_{i}, a \le a'_{i}}$, so $y \not \in K$, a contradiction.
\end{proof}

\begin{corollary}
	\label{2orbits}
$\mathfrak U$ is $\nicefrac{1}{2}$-homogeneous.
In particular, the orbits of the action of $\homeo(\mathfrak U)$ on $\mathfrak U$ are  $\mathfrak L \cap \mathfrak U$ and $\mathfrak M$.
\end{corollary}
\begin{proof}
By \Cref{homogeneityofcomponents}, for any $x, x' \in \mathfrak M$, $y, y' \in \mathfrak L \cap \mathfrak U$ distinct, there is $h \in \homeo_{\le}(\quotofpre)$ such that $h(x)=x', h(y)= y'$.
Since $h \restr{\mathfrak U} \in \homeo(\mathfrak U)$, it follows that there are at most $2$ orbits of the action of $\homeo(\mathfrak U)$ on $\mathfrak U$.
Therefore it suffices to show that $ \mathfrak U $ is not homogeneous.
By \Cref{maxarepolish} the space $ \mathfrak U $ is Polish, by \Cref{singlezerodimensional} it is not cohesive and by \Cref{notzerodim} it is not zero-dimensional.
By \cite{Dijkstra2006}*{Proposition 2}, a Polish, non-cohesive, non-zero-dimensional space is not homogeneous.
\end{proof}

\begin{proposition}
	\label{allendpointsaredense}
$\mathfrak M$ and  $\mathfrak L \cap \mathfrak U $ are dense in  $\quotofpre$.
\end{proposition}
\begin{proof}
It is easy too see that $\mathfrak M$ is dense in $\quotofpre$ by \Cref{weakcharacthm}.

To see that $\mathfrak L \cap \mathfrak U$ is dense, let $O$ be a nonempty open subset of $\quotofpre$ and let $n_{0} \in \N$, $a_{0} \in \fus_{n_{0}}$ be such that $\cappello{a_{0}}_{\fue_{n_0}}\subseteq O$.
We define a sequence $(a_{i})_{i \in \N}$ by induction.
Suppose that $n_{i}$ and $a_{i} \in \fus_{n_{i}}$ are defined and let $P_{i} = \fus_{n_{i}} \sqcup \set{b}$ and  $\fhi_{i}: P_{i} \to \fus_{n_{i}}$ be the identity on $\fus_{m}$ and $\fhi_{i}(b)=a_{i}$.
By \ref{itm:Lthreeprime} there are $n_{i+1}> n_{i}$ and an epimorphism $\psi_{i}: \fus_{n_{i+1}} \to P_{i}$ such that $\fhi_{i} \psi_{i} = \fue^{n_{i+1}}_{n_{i}}$.
By \Cref{lemchain}, there exists $B_{i} \in \branches(\fus_{n_{i+1}})$ such that $\psi_{i}[B_{i}] = \set{b}$, so $\fue^{n_{i+1}}_{n_{i}}[B_{i}] = \set{a_{i}}$.
Choose $a_{i+1} \in B_{i}$, so $\fue^{n_{i+1}}_{n_{i}}(a_{i+1})= a_{i}$.

Let $u \in \prespace$ be such that $\fue_{n_{i}}(u) = a_{i}$ for each $i \in \N$.
For each $i \in \N$, we have that $\fue_{n_{i+1}}(u) \in B_{i}$ and $\fue^{n_{i+1}}_{n_{i}} (\max B_{i}) = \fue^{n_{i+1}}_{n_{i}} (\min B_{i}) = a_{i} = \fue_{n_{i}}(u)$.
By \Cref{endpointprop}, $u$ is both $\le^{\prespace}$-minimal and $\le^{\prespace}$-maximal.
It follows that $p(u) \in \mathfrak L \cap \mathfrak U$.
Since $\fue_{n_{0}} (u) = a_{0}$, we have $p(u) \in \cappello{a_{0}}_{\fue_{n_0}}\subseteq O$.
\end{proof}

\begin{proposition}
	\label{notcohesive}
$\mathfrak M, \mathfrak U$ have the property that each nonempty open set contains a nonempty clopen subset.
In particular they are not cohesive.
\end{proposition}
\begin{proof}
The result for $\mathfrak U$ follows from Propositions \ref{allendpointsaredense} and \ref{singlezerodimensional}.

Let $O$ be an open subset of $\quotofpre$ such that $O \cap \mathfrak M \neq \emptyset$.
Up to taking a subset we can assume $O$ is $\le^{\prespace/R^{\prespace}}$-convex.
By \Cref{weakcharacthm} there exists an arc component $J$ of $\quotofpre$ whose endpoints both belong to $O$, so by  $\le^{\prespace/R^{\prespace}}$-convexity, $J \subseteq O$.
By \Cref{branchforcomponent} there exist $n \in \N$ and $B \in \branches(\fus_{n})$ such that  $J \subseteq \bigcup_{a \in B} \cappello{a}_{\fue_n}\subseteq O$.
Since $\bigcup_{a\in B} \cappello{a}_{\fue_n}$ is clopen in $ \quotofpre $ by \Cref{branchesareclopen}, it follows that $
\bigcup_{a\in B} \cappello{a}_{\fue_n}
\cap \mathfrak M$ is clopen in $ \mathfrak M $, and it is nonempty as it contains $\max J$.
\end{proof}

Finally we look at $\mathrm{E}(\quotofpre) = \mathfrak L \cup \mathfrak U$.

\begin{proposition}
	\label{nottotsep}
The spaces $\mathrm{E}(\quotofpre)$ and $\mathfrak L \symdif \mathfrak U $ are not totally separated.
In fact, in $\mathfrak L \symdif \mathfrak U $ the quasi-component of each point has cardinality $2$.
\end{proposition}
\begin{proof}
Let $x\in \mathfrak L \symdif \mathfrak U $, say $x\in \mathfrak M $ and let $z$ be the least element of the connected component $J$ of $x$ in $ \quotofpre $.
Let $U$ be a clopen neighborhood of $x$ in $ \mathfrak L \symdif \mathfrak U $, and let $O$ be open in $ \quotofpre $ such that $U=O\cap ( \mathfrak L \symdif \mathfrak U )$.

If $J\nsubseteq \closure (O)$, from the proof of \Cref{notzerodim} it follows that there exists some $y\in \closure_{ \mathfrak M } \left (
O\cap  \mathfrak M  \right ) \setminus O$, so
\begin{multline*}
\emptyset\ne \closure_{ \mathfrak M } \left (
O\cap  \mathfrak M  \right ) \setminus O\subseteq \closure_{ \mathfrak L \symdif \mathfrak U } \left (
O\cap  \mathfrak M  \right ) \setminus O\subseteq \\
\subseteq \closure_{ \mathfrak L \symdif \mathfrak U } (O\cap ( \mathfrak L \symdif \mathfrak U ))\setminus (O\cap ( \mathfrak L \symdif \mathfrak U ))=\partial_{ \mathfrak L \symdif \mathfrak U }(U),
\end{multline*}
contradicting the fact that $U$ is clopen in $ \mathfrak L \symdif \mathfrak U $.

If $J\subseteq \closure (O)$ but $z\notin O$, given any open neighborhood $V$ of $z$ in $ \quotofpre $, by \Cref{weakcharacthm} there is some $w\in  \mathfrak M \cap O\cap V$, so $w\in U\cap V$.
This implies that $z\in \closure_{ \mathfrak L \symdif \mathfrak U }(U)\setminus U$, contradicting again the fact that $U$ is clopen in $ \mathfrak L \symdif \mathfrak U $.

Therefore the intersection of all clopen neighborhoods of $x$ in $ \mathfrak L \symdif \mathfrak U $ also contains $z$.
On the other hand any two points belonging to distinct components of $\quotofpre$ can obviously be separated by clopen sets, so the quasi-component of $x$ in $\mathfrak L \symdif \mathfrak U $ is $ \set{x, z} $.
\end{proof}

Since almost zero-dimensional, $T_0$ spaces are totally separated, it follows that the spaces $ \mathfrak L \symdif \mathfrak U$ and $\mathrm{E}(\quotofpre)$ are not almost zero-dimensional.
This should be contrasted with \Cref{propluazd}.

We sum up what we know about the spaces of endpoints of the \F fence.

\begin{theorem}
	\label{riassunto}
\leavevmode
\begin{enumerate}[label=\upshape (\roman*)]
\item
	\label{singlearebaire}
$\mathfrak L \cap \mathfrak U$ is homeomorphic to the Baire space $\N^{\N}$.
\item
	\label{allendpoints}
$\mathrm{E}(\quotofpre)$ is Polish and not totally separated.
\item
	\label{maximalrecap}
$\mathfrak U$ is $\nicefrac{1}{2}$-homogeneous, Polish, almost zero-dimensional, one-dimensional and not cohesive.
\item
\label{asexample} \sloppy
$\mathfrak M$ is homogeneous, strongly $\sigma$-complete, almost zero-dimensional, one-dimensional and not cohesive.
\end{enumerate}
\end{theorem}

\begin{proof}
\leavevmode
\begin{enumerate}[label=\upshape (\roman*)]
\item
By \Cref{singlepolish} and \Cref{singlezerodimensional}, $\mathfrak L \cap \mathfrak U$ is Polish and zero-dimensional.
By \cite{MR1321597}*{Theorem 7.7} it is enough to show that every compact subset of $ \mathfrak L \cap \mathfrak U $ has empty interior.
So let $K$ be such set, and suppose toward contradiction that there is an open subset $O$ of $\mathfrak U $ such that $\emptyset \neq O\cap \mathfrak L \cap \mathfrak U =O \cap \mathfrak L \subseteq K$.
Recall that, by \Cref{allendpointsaredense}, $ \mathfrak L \cap \mathfrak U $ is dense and codense in $\mathfrak U$.
Then $O \setminus ( \mathfrak L \cap \mathfrak U )=O\setminus K$ is open in $ \mathfrak U $.
Therefore, by denseness of $ \mathfrak L \cap \mathfrak U $, it follows that $O\setminus ( \mathfrak L \cap \mathfrak U )=\emptyset $, contradicting codenseness.
\item
This holds by \Cref{maxarepolish} and \Cref{nottotsep}.
\item
This holds by \Cref{2orbits}, \Cref{maxarepolish}, and \Cref{propluazd,notzerodim,notcohesive}.
\item
This holds by \Cref{homogeneous}, \Cref{stronglysigmacomplete}, and \Cref{propluazd,notzerodim,notcohesive}.
\end{enumerate}
\end{proof}

A space with the properties listed in \ref{asexample} was first exhibited in \cite{Dijkstra2006} as a counterexample to a question by Dijkstra and van Mill.
We do not know however whether the two spaces are homeomorphic.

\begin{question}
Is $\mathfrak M$ homeomorphic to the space in \cite{Dijkstra2006}?
\end{question}

\end{document}